\documentclass{birkjour}
\usepackage{mathrsfs,amsfonts,amsmath,amsbsy}\usepackage{fontenc}\usepackage{textcomp}
\usepackage{graphicx}
\usepackage{amssymb}
\usepackage{graphicx}

\def\O{\Omega}

\def\f{\frac}

\def\le{\leq}

\def\bbb{\begin{eqnarray*}}

\def\eee{\end{eqnarray*}}

 \newtheorem{thm}{Theorem}[section]
 \newtheorem{cor}{Corollary}[section]
 \newtheorem{lem}{Lemma}[section]
 
 \theoremstyle{definition}
 \newtheorem{defn}{Definition}[section]
 \theoremstyle{remark}
 \newtheorem{rem}{Remark}[section]
 \newtheorem{ex}{Example}[section]
 \numberwithin{equation}{section}

\begin{document}

%
%
%
%
%
%
%
%
%

\title[Dependence of Discrete Sturm-Liouville Eigenvalues on Problems]
{Dependence of Discrete \\
 Sturm-Liouville Eigenvalues on Problems}


\author{Hao Zhu}

\address{%
Department of Mathematics, Shandong University\\
Jinan, Shandong 250100, P. R. China}

\email{haozhusdu@163.com}

\thanks{This research is supported by the NNSF of Shandong Province (Grant ZR2011AM002)
and the NNSF of China (Grant 11071143).}

\author{Shurong Sun}
\address{School of Mathematical Sciences, University of Jinan\\
Jinan,  Shandong 250022, P. R. China}
\email{sshrong@163.com}
\author{Yuming Shi}
\address{Department of Mathematics, Shandong University\\
Jinan, Shandong 250100, P. R. China}
\email{ymshi@sdu.edu.cn}
\author{Hongyou Wu}
\address{Department of Mathematics, Northern Illinois University\\
 DeKalb, IL 60115, USA}

\subjclass{ 39A12; 34B24; 39A70.}

\keywords{ discrete Sturm-Liouville problem;  dependence;
eigenvalue; continuous eigenvalue branch; self-adjoint problem.}

\date{September 21, 2014}
\dedicatory{Dedicated to the memory of the fourth author, Professor Hongyou Wu (1962.11--2009.6.9).}


\begin{abstract} This paper is concerned with dependence of  discrete Sturm-Liouville
eigenvalues on problems. Topologies and geometric structures on various spaces of
such problems are firstly introduced. Then, relationships between the analytic and
geometric multiplicities of an eigenvalue are discussed. It is shown that all problems
sufficiently close to a given problem have eigenvalues near each
eigenvalue of the given problem. So, all the simple eigenvalues live in so-called
continuous simple eigenvalue branches over the space of problems, and all the eigenvalues live in
continuous eigenvalue branches over the space of self-adjoint problems.
The analyticity, differentiability and monotonicity of continuous eigenvalue branches
are further studied.\end{abstract}

\maketitle
\section{Introduction}

A discrete Sturm-Liouville problem (briefly, SLP) considered in the present paper consists of a discrete Sturm-Liouville equation (briefly, SLE)
\begin{equation}\label{11}
-\nabla(f_{n}\Delta y_{n})+q_{n}y_{n}=\lambda w_{n}y_{n}, \;\;\;\; \ \  n\in[1,N],
\end{equation}
and the boundary condition (briefly, BC)
 \begin{equation}\label{12}
 A \begin{pmatrix}y_{0}\\
f_{0}\triangle y_{0}\end{pmatrix}
+B\begin{pmatrix}y_{N}\\
f_{N}\triangle y_{N}
\end{pmatrix}=0,
 \end{equation}
where $N \geq 2$ is an integer, $\Delta $ and $\nabla$ are the forward and backward difference
operators, respectively, i.e., $\Delta y_n=y_{n+1}-y_n$ and $\nabla y_n=y_n-y_{n-1}$;
$f=\{f_n\}_{n=0}^{N}$, $q=\{q_n\}_{n=1}^{N}$ and $w=\{w_n\}_{n=1}^{N}$ are complex-valued sequences such that

 \begin{equation}\label{13}
f_n\neq0 \;{\rm for}\;\; n\in [0,N],
\;\;w_n\neq0\;{\rm for}\; n\in [1,N];
\end{equation}
$\lambda$ is the spectral parameter; the interval $[M,N]$ denotes the set of integers
 $\{M,M+1,\cdots,N\}$; and $A$ and $B$ are
$2\times2$ complex matrices such that
 \begin{equation}\label{14}
 \text{ rank}(A,B)=2.
\end{equation}

 Throughout this paper, by $\mathbb{C}$, $\mathbb{R}$, and $\mathbb{Z}$ denote the sets of the complex
numbers, real numbers, and integer numbers, respectively; and by $\bar{z}$ denotes the
complex conjugate of $z\in \mathbb{C}$. Moreover, when a capital Latin letter stands for a
matrix, the entries of the matrix are denoted by the corresponding
lower case letter with two indices. For example, the entries of a
matrix $C$ are $c_{ij}$'s.

The dependence of the continuous Sturm-Liouville
eigenvalues  on the problems and its applications have been extensively studied
 (cf., \cite{Cao1, Cao2, Cao3, Eastham, Everitt, Kato, Kong1, Kong2, Kong3, Peng, Poeschel, Zettl}). In \cite{Kong3}, Kong and Zettl proved that
the eigenvalues of continuous SLPs depend not only continuously but also smoothly on problems and then gave an expression for the derivative of
the $n$-th eigenvalue with respect to a given parameter in the continuous SLP.
Later, they, together with Wu, gave a natural geometric structure on the space of BCs of continuous SLPs in \cite{Kong2}.
This structure is the base for studying the dependence of Sturm-Liouville
eigenvalues on the BCs.  In addition, they investigated the differentiability of continuous eigenvalue branches
based on this structure, and discussed the relationships between the algebraic and geometric multiplicities of
an eigenvalue.

Along another line, research on discrete spectral problems and their inverse problems
has been of growing interest in recent years (cf., e.g.\cite{Atkinson, Bohner, Clark1, Clark2, Clark3, Jirari, Lv, Shi1, Shi2, Shi3, Shi4, Sun, Wang}).
Atkinson \cite{Atkinson} and Jirari \cite{Jirari} studied spectral problems of second-order
discrete scalar self-adjoint SLPs with separate BCs.
In \cite{Shi3}, the third author of the present paper with her coauthor Chen  investigated the following vector difference equation
\begin{equation}\label{15}
-\nabla(C_{n}\Delta y_{n})+B_{n}y_{n}=\lambda w_{n}y_{n},\;\;n\in [1,N],\;\;\;N\geq2,
\end{equation}
with the general boundary condition
\begin{equation}\label{16}
R\begin{pmatrix}-y_{0}\\
y_{N}\end{pmatrix}
+S\begin{pmatrix}C_{0}\Delta y_{0}\\
C_{N}\Delta y_{N}
\end{pmatrix}=0,
\end{equation}
where $C_{n}(n\in[0,N])$, $B_{n}$, and $\omega_{n}$($n\in[1,N]$) are Hermitian $d\times d$ matrices, $C_{0}$ and $C_{N}$ are nonsingular,
$\omega_{n}>0$ for $n\in[1,N]$; R and S are $2d\times2d$ matrices with ${\rm rank}(R,S)=2d$. It is evident that
the BC \eqref{12} is included in the BC \eqref{16}. The spectral results obtained in \cite{Shi3} will be used
in the study of the multiplicity of eigenvalues in the present paper.
 Further,
 the third author of the present paper with her coauthor Lv studied error estimate
 of eigenvalues of perturbed problems, sufficiently close to a given Sturm-Liouville
 problem \eqref{15} and \eqref{16}, by some variational properties of the eigenvalues
 under a certain non-singularity condition in \cite{Lv}.
 So we obtained the continuous dependence of eigenvalues on problems under the nonsingularity condition.

In  Chapter 2 of \cite{Kato}, Kato investigated perturbation problems for linear operators in finite-dimensional spaces.
 He studied how the eigenvalues change with the operator, in particular when the operator analytically depends on a parameter.
  His method is based on function-theoretic study of the corresponding resolvent. Obviously, the eigenvalue problem of the self-adjoint
   discrete SLPs consisting of \eqref{11}-\eqref{12} corresponds to that of an operator in a finite-dimensional space. Note that the operator defined
   by \eqref{11}-\eqref{12} may be multi-valued since $x(0)$ and $x(N+1)$ may not be uniquely determined by the BC \eqref{12}, and the problem discussed in the
    present paper is dependent on multi-parameters. However, the operators  are all single-valued
and their perturbations are only referred to one single parameter in \cite{Kato}.
So the results in \cite{Kato} can not be directly available in our study.

In the present paper, we shall investigate dependence of eigenvalues
 on the SLP consisting of \eqref{11} and \eqref{12}. There are two main motivations for our study.
Firstly, it is helpful to clarify the common
features and differences between the class of continuous SLPs and
that of discrete SLPs.
Secondly, it is hoped that findings of such work will form a theoretical
foundation for numerical works on discrete SLPs and their inverse
problems, and such numerical works will shed light on numerical works
on continuous SLPs and their inverse problems.
Many results in the continuous case may be obtained from the
corresponding results in the discrete case, via certain limit
procedures, but not vice verse; while some results in the discrete
case have relatively direct proofs. In this way, shorter proofs of
results in the continuous case may be found.

This paper is organized as follows. In Section 2, we give topologies
and geometric structures on various spaces of discrete SLPs, which are
fundamental for further developments. In Section 3, we first discuss
properties of the analytic and geometric multiplicities of eigenvalues
of the discrete SLPs and their relationships, and then study continuous dependence of
eigenvalues on the problems. In Section 4, we investigate some fundamental properties of continuous
eigenvalue branches including their analyticity, differentiability and
monotonicity. Finally, several examples illustrating results of these sections are
presented in Section 5.

\begin{rem}\label{r11} We shall apply the results obtained in the present paper
to study some other topics about discrete Sturm-Liouville problems,
including dependence of the $n$-th eigenvalue on problems, inequalities among eigenvalues
for different problems, and index problems for eigenvalues in our forthcoming papers.\end{rem}

\section{Spaces of problems}

In this section, we shall first introduce the topologies and geometric structures on the spaces of discrete SLEs, BCs,
 and self-adjoint BCs, separately, and then give the geometric structures of the spaces of discrete SLPs and self-adjoint discrete SLPs.
  On the one hand, unlike in the continuous case, the space of
discrete SLEs in this paper has an easy and obvious
structure. On the other hand, the space of BCs and the space of
self-adjoint BCs have the same geometric structures as those in the
continuous case.

Let the discrete SLE \eqref{11} be abbreviated as $(1/f,q,w)$. Then the space of discrete SLEs can be written as
$$ \Omega_N^{\mathbb{C}}:=\{(1/f,q,w):\eqref{13} \;{\rm holds }\}         $$
and is equipped with the topology deduced from the complex space $ \mathbb{C}^{3N+1}$.
Bold faced lower case Greek letters, such as $ \pmb\omega$, are used to denote elements of $\Omega_N^{\mathbb{C}}$.
The subspace $ \Omega_N^{\mathbb{R}}$ of $\Omega_N^{\mathbb{C}}$ has its obvious meaning. For convenience, the maximum norm
on $ \mathbb{C}^{3N+1}$ will be used:
$$\|(1/f,q,w)\|=\max\left\{|1/f_0|,\max_{n\in [1,N]}\left\{|1/f_n|,|q_n|,|w_n|\right\}\right\}.$$
Note that $\Omega_N^{\mathbb{C}}$ is a connected open subset of $ \mathbb{C}^{3N+1}$. Similarly, $ \Omega_N^{\mathbb{R}}$ is
an open subset of $ \mathbb{R}^{3N+1}$ and has $2^{2N+1}$ connected components, two of which are
$$\begin{array}{rrll}\Omega_N^{-,\mathbb{R},+}:=\{(1/f,q,w)\in\Omega_N^{\mathbb{R}}:f_n<0\;{\rm for}\; n\in [0,N],w_n>0\;{\rm for}\; n\in [1,N]\},\\[1.0ex]
\Omega_N^{+,\mathbb{R},+}:=\{(1/f,q,w)\in\Omega_N^{\mathbb{R}}:f_n>0\;{\rm for}\; n\in [0,N],w_n>0\;{\rm for}\; n\in [1,N]\}.\end{array} $$
We also set
$$\begin{array}{rrll}\Omega_N^{\mathbb{R},+}&:=&\{(1/f,q,w)\in\Omega_N^{\mathbb{R}}:&w_n>0\;{\rm for}\; n\in [1,N]\},\end{array} $$
which has $2^{N+1}$ connected components.
\medskip

Since equivalent linear algebraic systems of the form \eqref{12} define the same BC following \cite{Kong2}, we will take the quotient space
\begin{equation}\label{21}
\mathcal A^{\mathbb{C}} :=\raise2pt\hbox{${\rm M}^*_{2,4}(\mathbb C)$}/\lower3pt\hbox{${\rm
GL}(2,\mathbb C)$},
\end{equation}
equipped with the quotient topology, as the space of BCs, where
\vspace{-0.05cm}$$\begin{array}{l}M_{2,4}^{*}(\mathbb{C}):=\{2\times4 \;{\rm complex\; matrix}\;(A,B):\eqref{14}{\rm \;holds}\},\\[1.0ex]
GL(2,\mathbb{C}):=\{2\times2\; {\rm comlplex \;matrix}\; T:{\rm det}\; T \neq0\}.\end{array}\vspace{-0.05cm} $$
Note that $M_{2,4}^{*}(\mathbb{C})$ is an open subspace of $\mathbb{C}^{2\times4}.$
\eqref{21} implies that $(A_1,B_1)\sim(A,B)$ if $(A_1,B_1)=T(A,B)$, $T \in GL(2,\mathbb{C})$. Each BC is an equivalence class of coefficient
 matrices of systems of the form \eqref{12}; that is, an element of $\mathcal{A}^{\mathbb{C}}.$
We use $[A\,|\,B]$ to denote the BC represented by the system \eqref{12}. Bold faced capital Latin letters, such as $\mathbf{A}$, are also used for BCs.

Note that the space of BCs is independent of the equation \eqref{11} either in the continuous
or in the discrete case, and so it has the same topology and geometric structure in the discrete case as that in the continuous case. For convenience,
we present them as follows. We refer to Theorems 3.1 and 3.3 in \cite{Kong2} for details.

\begin{thm} \label{th21} The space $\mathcal{A}^{\mathbb{C}}$ of BCs is a connected and compact complex manifold of
complex dimension 4, while the space $\mathcal{A}^{\mathbb{R}}$ of real BCs is a connected and compact real-analytic manifold of dimension 4.\end{thm}

In addition, $\mathcal{A}^{\mathbb{C}}$ has the following canonical atlas of local coordinate systems:
\begin{equation}\label{22}
\begin{array} {cccc} \mathcal{N}_{1,2}^{\mathbb{C}}=\left \{\left [\begin{array} {llll}1&0&b_{11}&b_{12}\\
0&1&b_{21}&b_{22}\end{array}  \right ]:\;b_{ij}\in\mathbb{C},\; i, j= 1, 2\right \},\vspace{3mm}\\
\mathcal{N}_{1,3}^{\mathbb{C}}=\left \{\left [\begin{array} {cccc}1&a_{12}&0&b_{12}\\
0&a_{22}&-1&b_{22}\end{array}  \right ]:\;a_{i2},b_{i2}\in\mathbb{C},\; i=1,2\right \},\vspace{3mm}\\
\mathcal{N}_{1,4}^{\mathbb{C}}=\left \{\left [\begin{array} {cccc}1&a_{12}&b_{11}&0\\
0&a_{22}&b_{21}&1\end{array}  \right ]:\; a_{i2},b_{i1}\in\mathbb{C},\; i=1,2\right \},\vspace{3mm}\\
\mathcal{N}_{2,3}^{\mathbb{C}}=\left \{\left [\begin{array} {cccc}a_{11}&-1&0&b_{12}\\
a_{21}&0&-1&b_{22}\end{array}  \right ]:\; a_{i1},b_{i2}\in\mathbb{C},\;i=1,2\right \},\vspace{3mm}\\
\mathcal{N}_{2,4}^{\mathbb{C}}=\left \{\left [\begin{array} {cccc}a_{11}&-1&b_{11}&0\\
a_{21}&0&b_{21}&1\end{array}  \right ]:\; a_{i1},b_{i1}\in\mathbb{C},\;i=1,2\right \},\vspace{3mm}\\
\mathcal{N}_{3,4}^{\mathbb{C}}=\left \{\left [\begin{array} {cccc}a_{11}&a_{12}&-1&0\\
a_{21}&a_{22}&0&-1\end{array}  \right ]:\; a_{ij}\in\mathbb{C},\;i,j=1,2\right \},\end{array}
\end{equation}
which are the so-called canonical coordinate systems on $\mathcal{A}^{\mathbb{C}}$. The map
$$\begin{array}{cccc} \varphi_{1,2}:&\mathcal{N}_{1,2}^{\mathbb{C}}&\rightarrow&\mathbb{C}^{4},\end{array} $$
$$\begin{array}{cccc}&\left [\begin{array} {llll}1&0&b_{11}&b_{12}\\
0&1&b_{21}&b_{22}\end{array}  \right ]&\mapsto&(b_{11}, b_{12}, b_{21}, b_{22})\end{array}, $$
  is homeomorphic, and
 the coefficient matrix of the BC $\mathbf{A}$ can be written as the form
$$ \begin{pmatrix}1&0&b_{11}&b_{12}\\
0&1&b_{21}&b_{22}\end{pmatrix},
$$
which is called the corresponding normalized form. There are
similar statements about $\mathcal{N}_{1,3}^{\mathbb{C}}$, $\mathcal{N}_{1,4}^{\mathbb{C}}$, $\mathcal{N}_{2,3}^{\mathbb{C}}$, $\mathcal{N}_{2,4}^{\mathbb{C}}$, $\mathcal{N}_{3,4}^{\mathbb{C}}$. One of $\varphi_{i,j}$, $i<j$, $1\leq i \leq3$, $2\leq j\leq4$, is called a coordinate chart on $\mathcal A^{\mathbb C}$.
The above discussion gives a differentiable structure on $\mathcal A^{\mathbb C}$.
 In addition, the space $\mathcal{A}^{\mathbb{R}}$ has a similar
atlas of canonical coordinate systems, given by \eqref{22} with $\mathbb{C}$ replaced by $\mathbb{R}$ everywhere. Using the canonical coordinate systems
on $\mathcal{A}^{\mathbb{C}}$ and $\mathcal{A}^{\mathbb{R}}$, it is easy to determine how close to each other any two given BCs are.

For a point $p$ in a differential manifold $M$, we denote by $\mathrm{T}_{p}M$ the tangent space of $M$ at $p$. Now we give descriptions of the tangent spaces
of $\mathcal{A}^{\mathbb{C}}$ and $\mathcal{A}^{\mathbb{R}}$. If $\mathbf{A}\in \mathcal{N}_{1,2}^{\mathbb{C}}$, then
\begin{equation}\label{23}
\mathrm{T}_{\mathbf{A}}\mathcal{A}^{\mathbb{C}}=\mathrm{T}_{\mathbf{A}}\mathcal{N}_{1,2}^{\mathbb{C}}=\left \{\begin{pmatrix}0&0&l_{11}&l_{12}\\
0&0&l_{21}&l_{22}\end{pmatrix}:\; l_{ij}\in\mathbb{C},i,j=1,2\right \};
\end{equation}
if  $\mathbf{A}\in \mathcal{N}_{1,3}^{\mathbb{C}}$, then
\begin{equation}\label{24}
\mathrm{T}_{\mathbf{A}}\mathcal{A}^{\mathbb{C}}=\mathrm{T}_{\mathbf{A}}\mathcal{N}_{1,3}^{\mathbb{C}}=\left \{\begin{pmatrix}0&h_{12}&0&l_{12}\\
0&h_{22}&0&l_{22}\end{pmatrix}:\; h_{i2},l_{i2}\in\mathbb{C},i=1,2\right \};
\end{equation}
etc. The tangent spaces of $\mathcal{A}^{\mathbb{R}}$ have similar descriptions.

\begin{defn}\label{de21}
\begin{itemize}
\item[{\rm (i)}]A BC $[A\,|\,B]$ is said to be self-adjoint if
$$AEA^{*}=
BEB^{*},
$$ where$$E:=\begin{pmatrix} 0&1\\
-1&0\end{pmatrix},$$
and $A^{*}$ denotes the complex conjugate transpose of $A$. We use $\mathcal{B}^\mathbb{C}$ to denote the space of self-adjoint BCs.
\item[{\rm (ii)}]A BC is said to be degenerated if it can be written as the form
 $$\begin{array} {llll} \left [\begin{array} {llll}1&0&0&0\\
0&1&0&0\end{array}  \right ]& {\rm or} &
  \left [\begin{array} {llll}0&0&1&0\\
0&0&0&1\end{array}  \right ].\end{array} $$
\item[{\rm (iii)}]A BC is said to be separated if it can be written as the form
$$\left [\begin{array} {llll} a_{11}&a_{12}&0&0\\
0&0&b_{21}&b_{22}\end{array}   \right ].$$
We use $\mathcal{B}_{ S}$ to denote the space of separated self-adjoint BCs.
\item[{\rm (iv)}]A BC is said to be coupled if it is neither separated nor degenerated.
\end{itemize}\end{defn}

\begin{rem}\label{r21} Note that $(A,B)$ satisfies \eqref{25} if and only if $(A_1,B_1):=(TA,TB)$ does, where $T\in GL(2,\mathbb {C})$.
Therefore, the self-adjointness is well-defined. Moreover, the definition of self-adjointness is equivalent to Definition 2.1 in \cite{Shi3}.\end{rem}

The following result gives the canonical forms of separated and coupled self-adjoint BCs, respectively. We refer to Theorem 10.4.3 in \cite{Zettl} for details.

\begin{lem} \label{l21} The separated self-adjoint BCs can be written as
\begin{equation}\label{25}
\mathbf{S}_{\alpha,\beta}:=\left [\begin{array} {cccc}\cos\alpha&-\sin\alpha&0&0\\
0&0&\cos\beta&-\sin\beta\end{array}  \right ],
\end{equation}
where $$\alpha\in[0,\pi),\beta\in (0,\pi];$$
and the coupled self-adjoint BCs can be written as
$$[e^{i\gamma}K\,|\,-I],$$
where
$$\gamma\in[0,\pi),\;\;K\in SL(2,\mathbb{R}):=\{2\times2 \;real\; matrix\; M:{\rm det}M=1\}.$$\end{lem}

\begin{rem} \label {r22} The coupled self-adjoint BCs together form an open subset of $\mathcal{B}^\mathbb{C}$ and
$$\left \{[e^{i\gamma}K\,|\,-I]:\gamma\in[0,\pi),K\in SL(2,\mathbb{R})\right \}=\mathcal{N}_{3,4}^\mathbb{C}\cap\mathcal{B}^\mathbb{C}.$$\end{rem}

The following result gives the topology and the geometric structure of $\mathcal{B}^\mathbb{C}$.
\begin{thm} \label{th22} The space $\mathcal{B}^\mathbb{C}$ equals the union of the following relative open subsets:
\begin{equation}\label{26}
\begin{array} {cccc} \mathcal{O}_{1,3}^{\mathbb{C}}=\left \{\left [\begin{array} {cccc}1&a_{12}&0&\bar{z}\\
0&z&-1&b_{22}\end{array}  \right ]:\; a_{12},b_{22}\in\mathbb{R},z\in \mathbb{C}\right \},\vspace{3mm}\\
 \mathcal{O}_{1,4}^{\mathbb{C}}=\left \{\left [\begin{array} {cccc}1&a_{12}&\bar{z}&0\\
0&z&b_{21}&1\end{array}  \right ]:\; a_{12},b_{21}\in\mathbb{R},z\in \mathbb{C}\right \},\vspace{3mm}\\
\mathcal{O}_{2,3}^{\mathbb{C}}=\left \{\left [\begin{array} {cccc}a_{11}&-1&0&\bar{z}\\
z&0&-1&b_{22}\end{array}  \right ]:\; a_{11},b_{22}\in\mathbb{R},z\in \mathbb{C}\right \},\vspace{3mm}\\
 \mathcal{O}_{2,4}^{\mathbb{C}}=\left \{\left [\begin{array} {cccc}a_{11}&-1&\bar{z}&0\\
z&0&b_{21}&1\end{array}  \right ]:\; a_{11},b_{21}\in\mathbb{R},z\in \mathbb{C}\right \}.\end{array}
\end{equation}
Moreover, $\mathcal{B}^\mathbb{C}$ is a connected and compact real-analytic manifold of dimension 4.\end{thm}

\begin{proof}
 Direct calculations yield that all BCs in $\mathcal{O}_{1,3}^{\mathbb{C}}$, $\mathcal{O}_{1,4}^{\mathbb{C}}$, $\mathcal{O}_{2,3}^{\mathbb{C}}$
 and $\mathcal{O}_{2,4}^{\mathbb{C}}$ are self-adjoint.
Evidently, all separated self-adjoint BCs are in these subsets by Lemma \ref{l21}. Moreover, by Lemma 3.18 in \cite{Peng}, every coupled self-adjoint BC also lies in these subsets. Thus,
$\mathcal{B}^\mathbb{C}$ is the union of these subsets.

As a canonical coordinate system on $\mathcal{A}^\mathbb{C}$, $\mathcal{N}_{1,3}^{\mathbb{C}}$ is open
in $\mathcal{A}^\mathbb{C}$. It can be easily verified by a direct calculation that
$$\mathcal{O}_{1,3}^{\mathbb{C}}=\mathcal{N}_{1,3}^{\mathbb{C}}\cap\mathcal{B}^\mathbb{C},$$
and hence $\mathcal{O}_{1,3}^{\mathbb{C}}$ is a relatively open set in $\mathcal{B}^\mathbb{C}$. Similarly,
it can be shown that $\mathcal{O}_{1,4}^{\mathbb{C}}$, $\mathcal{O}_{2,3}^{\mathbb{C}}$ and $\mathcal{O}_{2,4}^{\mathbb{C}}$ are relatively open sets
 in $\mathcal{B}^\mathbb{C}$. Since each of $\mathcal{O}_{1,3}^{\mathbb{C}}$, $\mathcal{O}_{1,4}^{\mathbb{C}}$, $\mathcal{O}_{2,3}^{\mathbb{C}}$
 and $\mathcal{O}_{2,4}^{\mathbb{C}}$ is connected and intersects the other three, $\mathcal{B}^\mathbb{C}$ is connected.
 The proof of the rest part is the same as that of Theorem 3.11 in \cite{Kong2}.\end{proof}

\begin{rem}\label{r23}
\begin{itemize}
\item[{\rm (i)}] There are similar statements (except the dimension) about the space $\mathcal{B}^\mathbb{R}$ of real self-adjoint BCs
to those about $\mathcal{B}^\mathbb{C}$ in Theorem \ref{th22}. Note that $\mathcal{B}^\mathbb{R}$ has dimension 3 and is a submanifold of
$\mathcal{A}^\mathbb{R}$ (see Theorem 3.9 in \cite{Kong2} for detailed discussion).
\item[{\rm (ii)}] Theorem \ref{th22} says that $\mathcal{O}_{1,3}^{\mathbb{C}}$, $\mathcal{O}_{1,4}^{\mathbb{C}}$, $\mathcal{O}_{2,3}^{\mathbb{C}}$,
and $\mathcal{O}_{2,4}^{\mathbb{C}}$ together form an atlas of local coordinate systems on $\mathcal{B}^\mathbb{C}$.
     If $\mathbf A\in\mathcal{O}_{1,3}^{\mathbb{C}}$, then the corresponding coordinate chart is given by
  $$\hat \varphi:\left [\begin{array} {cccc}1&a_{12}&0&\bar{z}\\
0&z&-1&b_{22}\end{array}  \right ]\to(a_{12},a,b,b_{22}), $$
      where $z=a+ib$, $a$, $b\in \mathbb R$.
       The normalized form of the coefficient matrix of  a BC in $\mathcal O_{1,3}^{\mathbb{C}}$ is given naturally. Others are given similarly.
     The above discussion gives a differentiable structure on $\mathcal B^{\mathbb C}$.        There is a similar statement about $\mathcal{B}^\mathbb{R}$.
\item[{\rm (iii)}] This result has been first mentioned in Theorem 3.11 in \cite{Cao3}, and a proof can be deduced from it.
\item[{\rm (iv)}] The tangent spaces of $\mathcal{B}^\mathbb{C}$ and $\mathcal{B}^\mathbb{R}$ can be similarly described as in \eqref{23}, \eqref{24}, etc.
Here we omit the details, but they are mentioned in Theorem \ref{th45} in Section 4 about $\mathcal{B}^\mathbb{C}$.
\end{itemize}\end{rem}

\begin{defn} \label{de22} The discrete SLP consisting of a discrete SLE $(1/f,q,w)$ and a BC $\mathbf{A}$ is said to be self-adjoint if $(1/f,q,w)$ is in
$\Omega_N^{\mathbb{R},+}$ and $\mathbf{A}$ is self-adjoint.\end{defn}

From the above discussions, we immediately deduce the following conclusions, which give the geometric structures of the spaces of
discrete SLPs and self-adjoint discrete SLPs, respectively.

\begin{thm} \label{th23} The space $\Omega_N^{\mathbb{C}}\times\mathcal{A}^\mathbb{C}$ of discrete SLPs is a connected complex manifold of complex dimension $3N+5$, while the space $\Omega_N^{\mathbb{R},+}\times\mathcal{B}^\mathbb{C}$ of self-adjoint discrete SLPs is a real-analytic manifold of dimension
$3N+5$ and has $2^{N+1}$ connected components.\end{thm}

\begin{rem} \label{r24}Note that the differentiable structure of the product space $\Omega_N^{\mathbb{C}}\times\mathcal{A}^\mathbb{C}$
can be given by that of $\Omega_N^{\mathbb{C}}$ and $\mathcal{A}^\mathbb{C}$ naturally. There is a similar statement
about $\Omega_N^{\mathbb{R},+}\times\mathcal{B}^\mathbb{C}$.\end{rem}

\section{Multiplicity of eigenvalues and continuous eigenvalue branches}

In this section, we shall first discuss properties and relationships of analytic and geometric multiplicities of eigenvalues and then study continuous dependence of eigenvalues on problems.
We shall point out that these relationships of the multiplicities
of eigenvalues are very important in the following investigations because continuous eigenvalue branches are defined according to the analytic multiplicity of eigenvalues (see Theorem \ref{th35} and Remark \ref{r32}), while the
study on their properties, such as differentiability and monotonicity, is related to the geometric multiplicity of
eigenvalues (see Section 4). These relationships allow us to simplify the discussion about the above properties.

\subsection{ Multiplicity of eigenvalues }

In this subsection, we shall first study properties of geometric and analytic multiplicities of eigenvalues of discrete SLPs,
separately, and then establish their relationships. Especially, we shall show that they are equal by a direct method if the problem is self-adjoint.

Let $l$ denote the following natural difference operator corresponding to equation \eqref{11}:
$$(ly)_{n}=\omega_{n}^{-1}\left (-\nabla(f_{n}\Delta y_{n})+q_{n}y_{n}\right),\;\;n\in[1,N],
$$
and let
$$l[0,N+1] =\left\{y =\{y_n \}_{n=0}^{N+1}:\; y_{n}\in \mathbb{C},0\leq n \leq N+1 \right\}.
$$

\begin{defn}\label{de31}

\begin{itemize}
\item[{\rm (i)}]A complex number $\lambda$ is called an eigenvalue of the discrete SLP \eqref{11}-\eqref{12} if there exists $y\in l[0,N+1]$
 which is non-trivial and solves
the problem. The non-trivial solution $y$ is called an eigenfunction corresponding to $\lambda$.

\item[{\rm (ii)}]The complex vector space spanned by the eigenfunctions for an eigenvalue is called the  eigenspace corresponding to the eigenvalue,
while the dimension of the eigenspace is called the geometric multiplicity of the eigenvalue.
\item[{\rm (iii)}]An eigenfunction $y\in l[0,N+1]$ corresponding to an eigenvalue of the self-adjoint discrete SLP is said to be normalized provided that
\vspace{-0.05cm}$$\sum_{n=1}^{N}\omega_{n}y_{n}\bar{y}_{n}=1.
$$
\end{itemize}
\end{defn}

\begin{rem}\label{r31}\begin{itemize}
\item[{\rm (i)}]A solution $y$ of \eqref{11} is said to be non-trivial if there exists $n \in [0,N+1]$ such that $y_n
\neq 0$.
\item[{\rm (ii)}]Since \eqref{11} has exactly 2 linearly independent solutions, the geometric multiplicity of each eigenvalue is either 1 or 2.
\end{itemize}\end{rem}

The following uniqueness of solutions
of initial value problems of \eqref{11} can be easily verified.

\begin{lem}\label{l31} Let $m\in [0,N]$, and $z_m$, $z_m^{[1]} \in
\mathbb{C}$. Then, for each $\lambda \in \mathbb{C}$, the discrete SLE \eqref{11} has a unique solution $y(\lambda) \in l[0,N+1]$ satisfying
$$ y_m(\lambda)
=z_m,\;\;f_m\Delta y_m(\lambda) =z_m^{[1]}.$$
Moreover, for each $n \in [0,N]$, $y_n(\lambda)$ and
$f_n\Delta y_n(\lambda)$ are polynomials of $\lambda$.\end{lem}

For each $\lambda\in\mathbb{C}$, let $\phi(\lambda)$ and $\psi(\lambda)$ be the solutions of \eqref{11}
satisfying the initial conditions, respectively,
\begin{equation}\label{31}
\phi_0(\lambda)=1, f_0\Delta\phi_0(\lambda)=0; \;\;
\psi_0(\lambda)=0, \ f_0\Delta\psi_0(\lambda)=1.
\end{equation}
Then, by Lemma \ref{l31} any solution of \eqref{11} is a linear
combination of $\phi(\lambda)$ and $\psi(\lambda)$. Set
\begin{equation}\label{32}
\Phi_n(\lambda)=\begin{pmatrix}\phi_n(\lambda)& \psi_n(\lambda)\\f_n\Delta\phi_n(\lambda)&f_n\Delta\psi_n(\lambda)\end{pmatrix},
\;\;n \in [0,N], \;\; \lambda\in\mathbb{C}.
\end{equation}
Equation \eqref{11} can be rewritten as
$$f_n \Delta y_n =[1+(q_n-\lambda w_n)/f_{n-1}] f_{n-1} \Delta
y_{n-1}+(q_n-\lambda w_n) y_{n-1},  \;\;\;  n\in[1,N].$$
So we have\begin{equation}\label{33}
\Phi_n(\lambda)
=\begin{pmatrix} 1 & 1/f_{n-1} \\
          q_n-\lambda w_n & 1+(q_n-\lambda w_n)/f_{n-1}
          \end{pmatrix}
 \Phi_{n-1}(\lambda), \;\;\ \; n \in[1,N].
\end{equation}
$\Phi_n(\lambda)$ is called the transfer matrix of \eqref{11}. By induction from \eqref{33},
the leading terms of $\phi_N (\lambda)$, $\psi_N (\lambda)$, $f_N\Delta\phi_N (\lambda)$,
and $f_N\Delta\psi_N (\lambda)$ as polynomials of $\lambda$ are
\begin{equation}\label{34}
\begin{array}{cccc}(-1)^{N-1}\left(\prod\limits_{i=1}^{N-1}({w_i}/{f_i})\right)
\lambda^{N-1},&(-1)^{N-1}\left(({1}/{f_0})\prod\limits_{i=1}^{N-1}({w_i}/{f_i})\right)
\lambda^{N-1},
\\[2.0ex](-1)^{N}\left(w_N\prod\limits_{i=1}^{N-1}({w_i}/{f_i})\right)
\lambda^{N},&(-1)^{N}\left(({w_N}/{f_0})\prod\limits_{i=1}^{N-1}({w_i}/{f_i})\right)
\lambda^{N},\end{array}
\end{equation}
respectively. It follows from \eqref{31} and \eqref{33} that
\begin{equation}\label{35}
\det \Phi_n(\lambda)=1,\;\; \;  n \in[0,N].
\end{equation}

 The following result says
that the transfer matrix $\Phi_N(\lambda)$
determines the eigenvalues of the problem for every BC.

\begin{lem} \label{l32} A number $\lambda \in
\mathbb{C}$ is an eigenvalue of the discrete SLP
\eqref{11}-\eqref{12} if and only if $\lambda$ is
a zero of the polynomial
\begin{equation}\label{36}
\Gamma(\lambda): =\det(A+B\Phi_N(\lambda)).
\end{equation}
Therefore, either all the complex numbers are eigenvalues of the problem or the
problem has only finitely many eigenvalues.
\end{lem}

\begin{proof} Fix a $\lambda\in \mathbb{C}$. Let $y(\lambda):=c_1 \phi(\lambda)+c_2\psi(\lambda)$, where $c_1,c_2\in \mathbb{C}$.
Inserting $y(\lambda)$ into the BC \eqref{12} yields
\begin{equation}\label{37}
(A+B\Phi_N(\lambda))\begin{pmatrix} c_{1}\\
c_{2}\end{pmatrix}=0.
\end{equation}
Then $y(\lambda)$ is a non-trivial solution of \eqref{11} and \eqref{12}, and hence $\lambda$ is an eigenvalue of the SLP
 if and only if the determinant of the coefficient matrix in \eqref{37} vanishes; that is, $\Gamma(\lambda)=0$. Moreover, by Lemma \ref{l31} $\Gamma(\lambda)$ is a polynomial of $\lambda$. This completes the proof.\end{proof}

\begin{defn}\label{de32}
\begin{itemize}
\item[{\rm (i)}]The polynomial function $ \Gamma$, unique up to a non-zero
constant multiple, is called the characteristic function of the discrete
SLP, for its importance.
\item[{\rm (ii)}]
The order of an eigenvalue as a zero of $\Gamma$ is called the analytic multiplicity (or simply just multiplicity) of the eigenvalue.
 An eigenvalue is said to be simple if it has multiplicity 1, while
an eigenvalue of multiplicity 2 is called a double eigenvalue.\end{itemize}
\end{defn}

The following result can be easily deduced from \eqref{35} and \eqref{36} via direct calculations. It is useful in some situations.

\begin{lem}\label{l33} The characteristic function of
the discrete SLP \eqref{11}-\eqref{12}  can be written as
$$\Gamma(\lambda) =\det A +\det B +G(\lambda),$$
where
$$G(\lambda):=c_{11}\phi_N(\lambda)+c_{12}\psi_N(\lambda)+c_{21}f_N\Delta\phi_N(\lambda)
+c_{22}f_N\Delta\psi_N(\lambda),$$
$$ C: =\begin{pmatrix} b_{11} & b_{21} \\ b_{12} &b_{22}\end{pmatrix}
\begin{pmatrix} a_{22} & -a_{21} \\ -a_{12} & a_{11}\end{pmatrix}.
$$
\end{lem}

\begin{thm}\label{th31} For each $\lambda
\in \mathbb{C}$, among all boundary conditions,
$[\Phi_N(\lambda)\,|\,-I]$ is the unique one that has $\lambda$
as an eigenvalue of geometric multiplicity $2$.\end{thm}

\begin{proof} From \eqref{37}, a complex BC $[A\,|\,B]$ has $\lambda$
as an eigenvalue of geometric multiplicity 2 if and only if
$A=-B\Phi_N(\lambda)$. It follows that
\begin{equation}\label{38}
(A,B)=-B(\Phi_N(\lambda),-I).
\end{equation}
Since rank$(A,B)\leq{\rm rank}B$ from \eqref{38}, $B$ is nonsingular by \eqref{14}. Thus, the only BC that has
 $\lambda$ as an eigenvalue of geometric multiplicity 2 is the one $[\Phi_N(\lambda)\,|\,-I]$. The proof is complete.\end{proof}

Now, we discuss relationships between the analytic and geometric multiplicities of an
eigenvalue of an SLP.

\begin{thm} \label{th32} The analytic multiplicity of an
eigenvalue is greater than or equal to its geometric
multiplicity.\end{thm}

\begin{proof} It suffices to prove that the analytic multiplicity of
any eigenvalue $\lambda_* $ of geometric multiplicity 2 is at
least 2 by (ii) of Remark \ref{r31}. By Theorem \ref{th31}, we only need to show that as an
eigenvalue for the BC $[\Phi_N(\lambda_*)\,|\,-I]$, $\lambda_*$ has
multiplicity $\geq 2$. Now, in this case it follows from Lemma \ref{l33} that the characteristic function is given by
\begin{equation}\label{39}
\begin{array}{ll} \Gamma(\lambda)
& =2-(f_N\Delta\psi_N(\lambda_*))\phi_N(\lambda)
                       +(f_N\Delta\phi_N(\lambda_*))\psi_N(\lambda)
\vspace{2mm}\\
& +\psi_N(\lambda_*)(f_N\Delta\phi_N(\lambda))
 -\phi_N(\lambda_*)(f_N\Delta\psi_N(\lambda)).
\end{array}
\end{equation}
By \eqref{35} we obtain that
\begin{equation}\label{310}
\begin{array}{ll}
\phi'_N(\lambda)
(f_N\Delta\psi_N(\lambda))
  +\phi_N(\lambda) (f_N\Delta\psi_N'(\lambda))\\[1.0ex]
  -\psi'_N(\lambda)(f_N\Delta\phi_N(\lambda))
  -\psi_N(\lambda)(f_N\Delta\phi'_N(\lambda)) \equiv0,\;\;\lambda\in \mathbb{C}.\end{array}
\end{equation}
Then, \eqref{39} and \eqref{310} together yield that
$\Gamma'(\lambda_*) =0$; that is, the analytic multiplicity of
$\lambda_*$ is at least 2. The proof is complete.\end{proof}

We shall remark that the analytic and geometric multiplicities of an eigenvalue are
not necessarily equal for an SLP in general, see Examples \ref{e51} and \ref{e52}. However,
we shall show that they are equal in the case that the SLP is self-adjoint.

Next, we consider the self-adjoint case. The self-adjoint SLP \eqref{11}-\eqref{12} can be written as \eqref{15}-\eqref{16} by
 setting $ d=1,C_n=f_n,B_n=q_n,$ $$
R=(R_1,R_2)=\begin{pmatrix}-a_{11}&b_{11}\\-a_{21}&b_{21}\end{pmatrix},\;\;
S=(S_1,S_2)=\begin{pmatrix}a_{12}&b_{12}\\a_{22}&b_{22}\end{pmatrix}.$$
Then
\begin{equation}\label{311}
r:=\text{ rank} (R_1+S_1C_0,S_2)=\text{ rank} \begin{pmatrix}-a_{11}+f_0a_{12}&b_{12}\\-a_{21}+f_0a_{22}&b_{22}\end{pmatrix}.
\end{equation}
Obviously, $0\leq r \leq 2$. Further, we have
\begin{equation}\label{312}
\kappa:=\text{ det} (R_1+S_1C_0,S_2)=a_{21}b_{12}-a_{11}b_{22}+f_0(a_{12}b_{22}-a_{22}b_{12}).
\end{equation}

The following result is a direct consequence of Theorem 4.1 in \cite{Shi3}.

\begin{lem}\label{l34} The sum of geometric multiplicities of all the eigenvalues of a self-adjoint
SLP {\rm (1.1)-(1.2)} is $N-2+r$, and moreover, all its eigenvalues are real.\end{lem}

The following result can be deduced from  Theorem 4.3 in \cite{Shi4}. We shall give an alternative and direct proof as follows.

\begin{thm}\label{th33} The analytic and geometric multiplicities of each eigenvalue of a self-adjoint SLP {\rm (1.1)-(1.2)} are the same.\end{thm}

\begin{proof} For convenience, by $\tau_{1}$ and $\tau_{2}$ denote the sum of the analytic multiplicities
and that of the geometric multiplicities of all the eigenvalues of the self-adjoint SLP, respectively,
by $\lambda_1, \cdots , \lambda_s$ denote the distinct eigenvalues of the SLP and
by $\tau^1_1,\cdots,\tau^s_1$ and  $\tau^1_2,\cdots,\tau^s_2$
denote the corresponding analytic and geometric multiplicities,  respectively,
where $\tau^1_1+\cdots+\tau^s_1=\tau_{1}$, $\tau^1_2+\cdots+\tau^s_2=\tau_{2}$.

The rest proof is divided into two steps.

{\bf Step 1.} $\tau_{1}$=$\tau_{2}$.

We divide the discussion into three cases.

{\bf Case 1.} $r=2$.

 By Lemma \ref{l34}, $\tau_{2}=N$. From \eqref{311}, we get that $\kappa\neq0.$
 By Lemma \ref{l33}, \eqref{34}, and \eqref{312} one can get that the leading term of $\Gamma(\lambda)$ as a polynomial of $\lambda$ is
 $$(-1)^{N+1}\left(({w_N}/{f_0})\prod\limits_{i=1}^{N-1}({w_i}/{f_i})\right)\kappa\lambda^{N},
 $$
and then $\tau_{1}=N$.
Hence, $\tau_{1}=\tau_{2}=N$.

{\bf Case 2.} $r=1$.

By Lemma \ref{l34}, $\tau_{2}=N-1$. It follows from \eqref{311} that $\kappa=0$, and then
 $$(-1)^{N+1}\left(({w_N}/{f_0})\prod\limits_{i=1}^{N-1}({w_i}/{f_i})\right)\kappa\lambda^{N}=0.$$
Thus, $\tau_{1}\leq N-1$ by Definition \ref{de32}. Further, we have that $\tau_{1}\geq\tau_{2}$ by Theorem \ref{th32}. Hence, $\tau_{1}=\tau_{2}=N-1$.

{\bf Case 3.} $r=0$.

By Lemma \ref{l34}, $\tau_{2}=N-2$. From \eqref{311}, we get that
\begin{equation}\label{313}
\begin{array}{cccc}a_{11}=f_{0}a_{12},&a_{21}=f_{0}a_{22},&b_{12}=b_{22}=0.\end{array}
\end{equation}
By inserting \eqref{313} into \eqref{12} and by \eqref{14}, the BC can be written as the form
 $$ \begin{array}{cc}y_{1}=0,&y_{N}=0.\end{array}
 $$
 This implies that there exists a $T\in {\it GL}(2,\mathbb{C})$ such that
 $$ T(A,B)=(A_1,B_1),$$
 where
 $$A_1=\begin{pmatrix}f_0&1\\0&0\end{pmatrix},\;
 B_1=\begin{pmatrix}0&0\\1&0\end{pmatrix}.$$
 By Lemma \ref{l33} we get that
 $$\Gamma_1(\lambda):=\det\left(A_1+B_1\Phi_N(\lambda)\right)
 =-\phi_N(\lambda)+f_0\psi_N(\lambda). $$
Then we have that
 $$\Gamma(\lambda)={\rm det}T\cdot \Gamma_1(\lambda)=\left(\det T\right)\left(-\phi_N(\lambda)+f_0\psi_N(\lambda)\right),
 $$
which, together with \eqref{34}, implies that the coefficients of the terms $\lambda^{N}$ and $\lambda^{N-1}$  of $\Gamma(\lambda)$ are equal to zero.
Thus, $\tau_{1}\leq N-2$ by Definition \ref{de32}. Again by Theorem \ref{th32} we have that $\tau_{1}\geq\tau_{2}$. Hence, $\tau_{1}=\tau_{2}=N-2$.

{\bf Step 2.} $\tau_1^i=\tau_2^i$ for $1\leq i \leq s$.

By Theorem \ref{th32}, $\tau_1^i\geq\tau_2^i$ for $1\leq i \leq s $. Suppose that there exists a $j$, $1\leq j \leq s $, such that $\tau_1^j>\tau_2^j$.
Then \vspace{-0.2cm}
 \begin{equation}\label{314}
\tau_1=\sum^{s}_{i=1}\tau^i_1>\sum^{s}_{i=1}\tau^i_2=\tau_2,
\end{equation}
 which contradicts to $\tau_{1}$=$\tau_{2}$.
 Therefore, the assertion holds and the entire proof is complete.\end{proof}
 The following result is a direct consequence of Theorems \ref{th31} and \ref{th33}. It has been firstly given in  Theorem 4.3.1 in \cite{Atkinson}
  for a special class of separated self-adjoint boundary conditions and
 then in  Theorem 2.2.6 in \cite{Jirari} for more general case.
\begin{cor}\label{c31} Assume that \eqref{11} is in
$\Omega_N^{\mathbb R,+}$. Then all the eigenvalues for each separated
self-adjoint boundary condition are simple.\end{cor}

\subsection{ Continuous eigenvalue branches}

In this subsection, using the topologies and geometric structures on the space of discrete SLPs
introduced in Section 2, we shall show that sufficiently close
discrete SLPs have near-by eigenvalues in a given bounded region of
$\mathbb{C}$, and explain how such eigenvalues form the so-called continuous
eigenvalue branches. In a general case, all the simple eigenvalues live in so-called
continuous simple eigenvalue branches over the space of problems.
However, we can get a better result in the self-adjoint case
that all the eigenvalues, which may be simple or not simple,  live in
continuous eigenvalue branches over the space of the problems.

To indicate the dependence of $\Phi_n(\lambda)$ on the discrete
SLE \eqref{11}, we write $\Phi_n(\lambda,\pmb\omega)$ with
$\pmb\omega =(1/f,q,w)\in \Omega_N^{\mathbb{C}}$. The following result can be deduced from \eqref{31} and \eqref{33}.

\begin{lem}\label{l35} Let $\pmb\omega \in \Omega_N^{\mathbb{C}}$. For
each $\varepsilon>0$, there is $\delta>0$ such that if $\pmb\sigma
\in \Omega_N^{\mathbb{C}}$ satisfies $\| \pmb\sigma -\pmb\omega \|
<\delta$, then
$$
\| \Phi_n(\lambda,\pmb\sigma) -\Phi_n(\lambda,\pmb\omega)
\|_1 <\varepsilon, \;\;\;\; \;\;n \in [0,N],  \;\;|\lambda| \leq
1/\varepsilon,$$
 where $\|\cdot\|_1$ is the maximum norm for the $2\times2$ matrix.
\end{lem}

Now, we are ready to prove the locally continuous dependence of
 eigenvalues on the corresponding discrete SLP.

\begin{thm}\label{th34} Let $\lambda_* \in \mathbb{C}$ be an eigenvalue of an
SLP $(\mathbf\omega,\mathbf {A}) \in
\Omega_N^{\mathbb{C}} \times \mathcal{A}^{\mathbb{C}}$ with multiplicity $m$,
$R$ a bounded open subset of $\mathbb{C}$ such that $\lambda_*\in R$, and $\lambda_*$ the only
eigenvalue of $(\mathbf\omega,\mathbf{A})$ in the closure $\bar{R}$ of $R$.
Then, there is a neighborhood $\mathcal{U}$ of
$(\mathbf\omega,\mathbf {A})$ in $\Omega_N^{\mathbb{C}} \times \mathcal A^{\mathbb{C}}$
such that each problem in $\mathcal U$ has exactly $m$ eigenvalues
in $\bar{R}$, counting multiplicity, and they all lie in $R$.\end{thm}

\begin{proof} To indicate the dependence of $\Gamma(\lambda)$ on
the SLP $(\pmb\omega,\mathbf{A})$, we write
$\Gamma_{(\pmb\omega,\mathbf{A})}(\lambda)$. Let $\mathcal N$ be a
coordinate system in \eqref{22} containing $\mathbf{A}$. For all BCs in
$\mathcal N$, we compute the characteristic function using the
corresponding normalized form of the coefficient matrices of the
BCs. By Lemma \ref{l35}, when $(\pmb\sigma,\mathbf{B}) \in
\Omega_N^{\mathbb{C}} \times \mathcal A^{\mathbb{C}}$ is sufficiently close to
$(\pmb\omega,\mathbf{A})$, $\mathbf{B}$ is also in $\mathcal N$, and
$\Gamma_{(\pmb\sigma,\mathbf{B})}(\lambda)$ is close to
$\Gamma_{(\pmb\omega,\mathbf{A})}(\lambda)$  on $\bar{R}$.
Since $\Gamma_{(\pmb\omega,\mathbf{A})}(\lambda)$ (or $\Gamma_{\pmb\sigma,\mathbf{B}}(\lambda)$)
is a polynomial of $\lambda$ and the degree of $\Gamma_{(\pmb\omega,\mathbf{A})}(\lambda)$ (or $\Gamma_{(\pmb\sigma,\mathbf{B})}(\lambda)$) in $\lambda$
is less than or equal to $N$
by Lemma \ref{l33} and \eqref{34}, we can set $$
\Gamma_{(\pmb\omega,\mathbf{A})}(\lambda)=a_k\lambda^k+\cdots
+a_1\lambda+a_0,$$
$$\Gamma_{(\pmb\sigma,\mathbf{B})}(\lambda)=a_k\lambda^k+\cdots+a_1\lambda+a_0
+\varepsilon_N(\pmb\sigma,\mathbf{B})\lambda^N+\cdots+
\varepsilon_1(\pmb\sigma,\mathbf{B})\lambda+\varepsilon_0(\pmb\sigma,\mathbf{B}),
$$
where $k \leq N$, $(\pmb\sigma,\mathbf{B})$ is sufficiently close to $(\pmb\omega,\mathbf{A})$,
and the value of $\varepsilon_i(\pmb\sigma,\mathbf{B})\in\mathbb{C}$ is dependent on
$(\pmb\sigma,\mathbf{B}) \in \Omega_N^{\mathbb{C}} \times \mathcal A^{\mathbb{C}}$, $0\leq i\leq N$.
Since the boundary set $\partial \bar{R}$ is a compact subset of $\mathbb{C}$ and $\Gamma_{(\omega,\mathbf{A})}(\lambda)\neq0$ for all $\lambda\in \partial \bar{R}$,
there exists $\lambda_{0}\in \partial \bar{R}$ such that
\begin{equation}\label{315}
\inf\limits_{\lambda\in\partial \bar{R}}|\Gamma_{(\pmb\omega,\mathbf{A})}(\lambda)|=|\Gamma_{(\pmb\omega,\mathbf{A})}(\lambda_0)|
=:\eta>0.
\end{equation}
One can choose $\varepsilon>0$ satisfying that
\begin{equation}\label{316}
\varepsilon\cdot\sup\limits_{\lambda\in\partial \bar{R}}(|\lambda|^N+\cdots+|\lambda|+1)<\eta.
\end{equation}
Since
$\Gamma_{(\pmb\sigma,\mathbf{B})}(\lambda)\to \Gamma_{(\pmb\omega,\mathbf{A})}(\lambda)$
uniformly for $\lambda\in \bar{R}$
as $(\pmb\sigma,\mathbf{B})\to(\pmb\omega,\mathbf{A})$, there exists a neighborhood
$\mathcal {U}$ of
$(\pmb\omega,\mathbf {A})$ in $\Omega_N^{\mathbb{C}} \times \mathcal A^{\mathbb{C}}$
 such that $$
|\varepsilon_i(\pmb\sigma,\mathbf{B})|<\varepsilon,\;\;\;\;0\leq i\leq N,$$
which, together with \eqref{316}, yields that $$
|\varepsilon_N(\pmb\sigma,\mathbf{B})\lambda^N+\cdots+
\varepsilon_1(\pmb\sigma,\mathbf{B})\lambda+\varepsilon_0(\pmb\sigma,\mathbf{B})|<\eta,
\;\;\;\;\;\lambda\in\partial \bar{R}.$$
Therefore, we have by \eqref{315} that  $$
|\Gamma_{(\pmb\omega,\mathbf{A})}(\lambda)|>|\varepsilon_N(\pmb\sigma,\mathbf{B})
\lambda^N+\cdots+
\varepsilon_1(\pmb\sigma,\mathbf{B})\lambda+\varepsilon_0(\pmb\sigma,\mathbf{B})|,
\;\;\;\;\;\lambda\in\partial \bar{R}.$$
By Rouche's Theorem in complex analysis, $\Gamma_{(\pmb\sigma,\mathbf
B)}(\lambda)$ and $\Gamma_{(\pmb\omega,\mathbf{A})}(\lambda)$ have the
same number of zeros in $R$, counting order. Therefore, the SLP $(\pmb\sigma,\mathbf{B}) \in \mathcal{U}$ has exactly $m$ eigenvalues in
$\bar{R}$, counting multiplicity, and they all lie in
$R$. This proof is complete.\end{proof}

The following
result is a direct consequence of Theorem \ref{th34}.

\begin{cor}\label{c32} For each $m \in \mathbb N$, the set of discrete
SLPs having at least $m$
eigenvalues, counting multiplicity, is open in $\Omega_N^{\mathbb C}
\times \mathcal A^{\mathbb C}$.\end{cor}

Combining the reality of the eigenvalues for a self-adjoint discrete SLP and Theorem \ref{th34} yields the following result:

\begin{cor} \label{c33} Let $\lambda_* \in \mathbb{R}$ be an eigenvalue of a
discrete SLP $(\mathbf\omega,\mathbf {A}) \in
\Omega_N^{\mathbb{R,+}} \times \mathcal{B}^{\mathbb{C}}$ with multiplicity $m$,
$(r_1,r_2)$ a bounded open interval of $\mathbb{R}$ such that $\lambda_*\in (r_1,r_2)$, and $\lambda_*$ the only
eigenvalue of $(\mathbf\omega,\mathbf{A})$ in the close interval $[r_1,r_2]$.
Then, there is a neighborhood $\mathcal{U}$ of
$(\mathbf\omega,\mathbf {A})$ in $\Omega_N^{\mathbb{R,+}} \times \mathcal B^{\mathbb{C}}$
such that each problem in $\mathcal U$ has exactly $m$ eigenvalues
in $[r_1,r_2]$, counting multiplicity, and they all lie in $(r_1,r_2)$.\end{cor}

Based on the above discussion, we now form the continuous eigenvalue branches over $\Omega_N^{\mathbb C} \times
\mathcal A^{\mathbb C}$ or $\Omega_N^{\mathbb R,+} \times\mathcal B^{\mathbb C}$ through a fixed eigenvalue.\medskip

\begin{thm}\label{th35}
\begin{itemize}
\item[{\rm (1)}]
 Let $\lambda_*$ be a simple eigenvalue for a
discrete SLP $(\pmb\omega,\mathbf A) \in \Omega_N^{\mathbb C} \times
\mathcal A^{\mathbb C}$. Then there is a continuous
function $\Lambda: \mathcal M \to \mathbb C$ defined on a connected
neighborhood $\mathcal M$ of $(\pmb\omega,\mathbf A)$ in
$\Omega_N^{\mathbb C} \times \mathcal A^{\mathbb C}$ such that
\begin{itemize}\vspace{0.1cm}
\item[{\rm (i)}] $\Lambda(\pmb\omega,\mathbf A)=\lambda_*$;\vspace{0.1cm}
\item[{\rm (ii)}] for any $(\pmb\sigma,\mathbf B) \in \mathcal M$,\vspace{0.1cm}
$\Lambda(\pmb\sigma,\mathbf B)$ is a simple eigenvalue of
$(\pmb\sigma,\mathbf B)$.
\end{itemize}
\item[{\rm (2)}]
  Let $\lambda_*$ be an eigenvalue of a
self-adjoint discrete SLP $(\pmb\omega,\mathbf A)$ with multiplicity $2$.
Fix a small $\epsilon>0$ such that $\lambda_*$ is the only
eigenvalue of $(\pmb\omega,\mathbf A)$ in the interval
$[\lambda_*-\epsilon, \lambda_*+\epsilon]$. Then, there are continuous functions $\Lambda_1$,
$\Lambda_2:\mathcal F \to \mathbb R$ defined on a connected neighborhood $\mathcal F $ of $(\pmb\omega,\mathbf A)$
 in $\Omega_N^{\mathbb R,+} \times\mathcal B^{\mathbb C}$ such that
\begin{itemize}\vspace{0.1cm}
\item[{\rm (i)}] $\Lambda_1(\pmb\omega,\mathbf A) =\Lambda_2(\pmb\omega,\mathbf A)
=\lambda_*$;\vspace{0.1cm}
\item[{\rm (ii)}] $\lambda_*-\epsilon <\Lambda_1(\pmb\sigma,\mathbf B) \leq
\Lambda_2(\pmb\sigma,\mathbf B) <\lambda_*+\epsilon$ for each
$(\pmb\sigma,\mathbf B) \in \mathcal F$;\vspace{0.1cm}
\item[{\rm (iii)}] for every $(\pmb\sigma,\mathbf B) \in \mathcal F$,
$\Lambda_1(\pmb\sigma,\mathbf B)$ and
$\Lambda_2(\pmb\sigma,\mathbf B)$ are eigenvalues of
$(\pmb\sigma,\mathbf B)$.
\end{itemize}
\end{itemize}
\end{thm}
\begin{proof} Assertion (1) can be straightforward shown by Theorem \ref{th34}.

Now, we show assertion (2). By Corollary \ref{c33} there is a neighborhood $\mathcal{F}$ of
$(\mathbf\omega,\mathbf {A})$ in $\Omega_N^{\mathbb{R,+}} \times \mathcal B^{\mathbb{C}}$
such that each problem $(\pmb\sigma,\mathbf B)$ in $\mathcal F$ has exactly $2$ eigenvalues,
which are denoted by $\lambda_1(\pmb\sigma,\mathbf B)$ and $
\lambda_2(\pmb\sigma,\mathbf B)$ with $\lambda_1(\pmb\sigma,\mathbf B)\leq\lambda_2(\pmb\sigma,\mathbf B)$, respectively,
 in $[\lambda_*-\epsilon,\lambda_*+\epsilon]$ and they all lie in $(\lambda_*-\epsilon,\lambda_*+\epsilon)$.
  Note that $\mathcal{F}$ can be chosen such that it belongs to a connected component of $\Omega_N^{\mathbb{R,+}} \times \mathcal B^{\mathbb{C}}$
   by Theorem \ref{th23}.
Then one can define two functions $\Lambda_1$,
$\Lambda_2:\mathcal F \to \mathbb R$ such that $\Lambda_1(\pmb\sigma,\mathbf B)=\lambda_1(\pmb\sigma,\mathbf B)$ and $\Lambda_2(\pmb\sigma,\mathbf B)=\lambda_2(\pmb\sigma,\mathbf B)$. It can be easily verified
that $\Lambda_1$ and $\Lambda_2$ satisfy (i)-(iii).

Next, we prove that $\Lambda_1$,
$\Lambda_2:\mathcal F \to \mathbb R$ are continuous functions. Fix a $(\pmb\sigma,\mathbf B)\in \mathcal F$. If $\Lambda_1(\pmb\sigma,\mathbf B)=\Lambda_2(\pmb\sigma,\mathbf B)$,
 there exists a neighborhood $\mathcal U_{(r_1,r_2)}$ of $(\pmb\sigma,\mathbf B)$ in $\Omega_N^{\mathbb R,+} \times\mathcal B^{\mathbb C}$ such that
$\Lambda_1(\pmb\tau,\mathbf C), \Lambda_2(\pmb\tau,\mathbf C)\in (r_1,r_2)$ for each open interval $(r_1,r_2)$
satisfying $\Lambda_1(\pmb\sigma,\mathbf B)\in(r_1,r_2)\subset[\lambda_*-\epsilon, \lambda_*+\epsilon]$
 and each $(\pmb\tau,\mathbf C)\in\mathcal U_{(r_1,r_2)}$ by Corollary \ref{c33}. Now, we assume that
 $\Lambda_1(\pmb\sigma,\mathbf B)<\Lambda_2(\pmb\sigma,\mathbf B)$. Then, there exists a positive number $\epsilon_1$
 such that $$ \bigcup\limits_{i=1}^{2}(\Lambda_i(\pmb\sigma,\mathbf B)-\epsilon_1,\Lambda_i(\pmb\sigma,\mathbf B)+\epsilon_1)\subset(\lambda_*-\epsilon,\lambda_*+\epsilon) $$ and
\vspace{-0.1cm}$$\bigcap\limits_{i=1}^2(\Lambda_i(\pmb\sigma,\mathbf B)-\epsilon_1,\Lambda_i(\pmb\sigma,\mathbf B)+\epsilon_1)=\emptyset.$$
 By Corollary \ref{c33} there exists a neighborhood $\mathcal U_{\delta}$ of $(\pmb\sigma,\mathbf B)$ in $\Omega_N^{\mathbb R,+} \times\mathcal B^{\mathbb C}$ such that $\Lambda_1(\pmb\tau,\mathbf C)\in(\Lambda_1(\pmb\sigma,\mathbf B)-\delta, \Lambda_1(\pmb\sigma,\mathbf B)+\delta)$, $\Lambda_2(\pmb\tau,\mathbf C)\in(\Lambda_2(\pmb\sigma,\mathbf B)-\delta, \Lambda_2(\pmb\sigma,\mathbf B)+\delta)$ for each $\delta$ satisfying $0<\delta<\epsilon_1$ and each $(\pmb\tau,\mathbf C)\in\mathcal U_{\delta} $. Therefore, $\Lambda_1$ and  $\Lambda_2$ are continuous in $\mathcal{F}$.
This completes the proof.
\end{proof}

\begin{rem}\label{r32}
\begin{itemize}\vspace{-0.2cm}
\item[{\rm (1)}]In the case of general discrete SLP, that is, $(\pmb\omega,\mathbf A)\in\Omega_N^{\mathbb C} \times
\mathcal A^{\mathbb C}$, any two such functions as $\Lambda$ in (1) of Theorem \ref{th35} agree on the common part of their
domains, which is still a neighborhood of $(\pmb\omega,\mathbf A)$ in
$\Omega_N^{\mathbb C} \times \mathcal A^{\mathbb C}$. So, by the
continuous simple eigenvalue branch over $\Omega_N^{\mathbb C} \times
\mathcal A^{\mathbb C}$ through $\lambda_*$, we mean any such function.
In the case of self-adjoint discrete SLP, that is, $(\pmb\omega,\mathbf A)\in\Omega_N^{\mathbb R,+} \times
\mathcal B^{\mathbb C}$, and $\lambda_*$ is an eigenvalue with multiplicity $2$, there are
actually different functions on $\mathcal F$. Locally they are the
only such functions, to be called the continuous eigenvalue
branches over $\Omega_N^{\mathbb R,+} \times \mathcal B^{\mathbb C}$ through
$\lambda_*$.
\item[{\rm (2)}] Statement (1) of Theorem \ref{th35} holds if we replace $\Omega_N^{\mathbb C} \times\mathcal A^{\mathbb C}$
and $\mathbb C$ by $\Omega_N^{\mathbb R,+} \times\mathcal B^{\mathbb C}$ and $\mathbb R$, respectively.
 This gives the continuous simple eigenvalue branch over $\Omega_N^{\mathbb R,+} \times\mathcal B^{\mathbb C}$ through $\lambda_*$.
\item[{\rm (3)}] There are similar results for subspaces of $\Omega_N^{\mathbb C}
\times \mathcal A^{\mathbb C}$, such as $\Omega_N^{\mathbb C} \times \mathcal
A^{\mathbb R}$ and $\Omega_N^{\mathbb R} \times \mathcal A^{\mathbb C}$ to (1) of
Theorem \ref{th35}. There are similar results for subspaces of
$\Omega_N^{\mathbb R,+} \times \mathcal B^{\mathbb C}$, such as
$\Omega_N^{\mathbb R,+} \times \mathcal B^{\mathbb R}$ to (2) of
Theorem \ref{th35} and (2) of Remark \ref{r32}.
\item[{\rm (4)}] The third author of the present paper, together with Lv, obtained the continuous dependence of eigenvalues on problems
under a non-singularity condition (see Theorem3.1 and Corollary 3.1 in \cite{Lv}).
\end{itemize}
\end{rem}
The following result can be directly obtained by Theorem \ref{th32} and (1) of Theorem \ref{th35}.

\begin{cor} \label{c34} Let $\lambda_*$ be a simple eigenvalue
of a discrete SLP $(\pmb\omega,\mathbf
A)\in\Omega_N^{\mathbb C} \times \mathcal A^{\mathbb C}$, and $\Lambda$ the  continuous simple
eigenvalue branch over $\Omega_N^{\mathbb C} \times \mathcal A^{\mathbb C}$
through $\lambda_*$. Then, there exists a connected neighborhood $\mathcal M$ of $(\pmb\omega,\mathbf A)$ such that for each $(\pmb\sigma,\mathbf B) \in
\mathcal M$, $\Lambda(\pmb\sigma,\mathbf B)$ has
geometric multiplicity 1.\end{cor}

\section{Analyticity, differentiability, and monotonicity}

In this section, we shall investigate analyticity and differentiability
of continuous eigenvalue branches under some assumptions on their
multiplicities, and then study monotonicity of continuous eigenvalue branches
of self-adjoint discrete SLPs on boundary conditions and equations, separately.

\subsection{ Analyticity and differentiability of continuous simple eigenvalue branches }

In this subsection, we shall study the analyticity and differentiability of continuous
simple eigenvalue branches. To do this, we need the following two lemmas (see Theorem 2.1.2 in \cite{Hormander} and Chapter V in \cite{Robinson}, separately):

 \begin{lem}\label{l41} Let $f_j(w,z)$, $1\leq j \leq m$, be analytic functions of $(w,z)=(w_1,\cdots,w_m,$ $z_1,\cdots,z_n)$
 in a neighborhood of a point $(w^0,z^0)$ in $\mathbb C^m\times\mathbb C^n$, and assume that $f_j(w^0,z^0)=0$, $1\leq j \leq m$, and$$
\det(\partial f_j/\partial w_k)_{j,k=1}^m\big|_{(w^0,z^0)}\neq 0.
$$
Then the equations $f_j(w,z)=0$, $1\leq j \leq m$, have a uniquely determined analytic solution $w(z)$ in a neighborhood of $z^0$ such that
$w(z^0)=w^0$. Moreover, the derivative formula in the neighborhood of $z^0$ is determined by
\begin{equation}\label{41}
\sum_{k=1}^{m}({\partial f_{j}}/{\partial w_k})\,dw_k+
\sum_{i=1}^{n}({\partial f_j}/{\partial z_i})\,dz_i=0,\;\;1\leq j \leq m.
\end{equation}
\end{lem}

\begin{lem}\label{l42} Assume that $U\subset\mathbb{R}^{n+1}$ is an open set and $F:U\rightarrow\mathbb{R}$ is a $C^{r}$ function for some $r\geq 1$.
For $p\in\mathbb{R}^{n+1}$, we write $p=(x,y)$ with $x\in \mathbb{R}^{n}$ and $y\in \mathbb{R}$. Assume that $(x_0,y_0)\in U$ and
$$({\partial F}/{\partial y})(x_0,y_0)\neq 0.$$
Let $C=F(x_0,y_0)\in \mathbb{R}$. Then, there are open sets $V$ containing $x_0$ and $W$ containing $y_0$ with $V\times W\subset U$, and a $C^{r}$ function $h:V\rightarrow W$ such that $h(x_0)=y_0$ and
$$F(x,h(x))=C\;\;\;\;\;\; for\;\; all \;\;x\in V.$$
Further, for each $x\in V$, $h(x)$ is the unique $y\in W$ such that $F(x,y)=C$.\end{lem}

\begin{thm}\label{th41} Let $\lambda_* \in \mathbb C$ be a simple eigenvalue of a discrete SLP $(\pmb \omega,\mathbf A)\in \Omega_N^{\mathbb C}\times\mathcal A^{\mathbb C}$.
Then, the continuous simple eigenvalue branch $\Lambda$ defined on a neighborhood $\mathcal{F}$ of $(\pmb \omega,\mathbf A)$ in $\Omega_N^{\mathbb C}\times\mathcal A^{\mathbb C}$
 through $\lambda_*$ is analytic.
For a fixed discrete SLE, the derivative of $\Lambda$ at
$\mathbf A=[A\,|\,B]$ is given by
\begin{equation}\label{42}
d\Lambda\big|_{\mathbf A}(H\,|\,L)
=-\sum^2_{j,k=1}(d_{jk}h_{jk}+e_{jk}\l_{jk}) \Big/
  G'(\lambda_*), \;\;\;\;(H\,|\,L) \in {\rm T}_{\mathbf A} \mathcal A^{\mathbb C},
\end{equation}
where the coefficient matrices $D=(d_{jk})$ and $E=(e_{jk})$ are defined by
$$
D:=\begin{pmatrix} a_{22} & -a_{21} \\ -a_{12} & a_{11} \end{pmatrix}
  +\begin{pmatrix} b_{22} & -b_{21} \\ -b_{12} & b_{11}
  \end{pmatrix}
  \begin{pmatrix} f_N\Delta\psi_N(\lambda_*) & -f_N\Delta\phi_N(\lambda_*)
         \\ -\psi_N(\lambda_*) & \phi_N(\lambda_*) \end{pmatrix},
$$
\vspace{-0.1cm}$$
E:=\begin{pmatrix} b_{22} & -b_{21} \\ -b_{12} & b_{11}
\end{pmatrix}
  +\begin{pmatrix} a_{22} & -a_{21} \\ -a_{12} & a_{11}
\end{pmatrix}
   \begin{pmatrix} \phi_N(\lambda_*) & f_N\Delta\phi_N(\lambda_*)
         \\ \psi_N(\lambda_*) & f_N\Delta\psi_N(\lambda_*) \end{pmatrix}.
$$
\end{thm}

\begin{proof} For the fixed problem $(\pmb \omega,\mathbf A)$, we assume that
$$\pmb \omega=(1/f,q,w)\in \Omega_N^{\mathbb C},\;\; \;\;\mathbf A =\left [\begin{array} {llll}1&0&b_{11}&b_{12}\\
0&1&b_{21}&b_{22}\end{array}  \right ]\in \mathcal N_{1,2}^{\mathbb{C}}.$$
 For all BCs in
$\mathcal N_{1,2}^{\mathbb{C}}$, we compute $\Gamma$ using the
corresponding normalized form of the coefficient matrices of the
BCs.
Define
$$\begin{array}{cccc}\tilde \varphi:&\Omega_N^{\mathbb C}\times\mathcal N_{1,2}^{\mathbb{C}}&\rightarrow&\mathbb{C}^{3N+5},\end{array}$$
$$\begin{array}{cccc}(1/f',q',w')\times\left [\begin{array} {llll}1&0&b'_{11}&b'_{12}\\
0&1&b'_{21}&b'_{22}\end{array}  \right ]&\mapsto&(1/f',q',w',b'_{11}, b'_{12}, b'_{21}, b'_{22}).\end{array}
$$
Then $\tilde \varphi$ is a coordinate chart on $\Omega_{N}^{\mathbb C}\times\mathcal A^{\mathbb C}$.
For convenience, we set
 $$\mathcal{V}={\tilde \varphi}\left((\Omega_N^{\mathbb C}\times\mathcal N_{1,2}^{\mathbb{C}})
 \cap\mathcal{F}\right), $$
  $$\vspace{-0.1cm}\mathcal{K}:=\{(\lambda,1/f',q',w',b'_{11}, b'_{12}, b'_{21}, b'_{22}):\lambda\in\mathbb{C},(1/f',q',w',b'_{11}, \cdots, b'_{22})\in\mathcal{V}\}.$$

Now, consider $\Gamma$ restricted to the region $\mathcal{K}$.
By \eqref{33} and Lemma \ref{l32}, $\Gamma$ is a polynomial and hence an analytic function of all variables in $\mathcal{K}$.
Since $\lambda_*=\Lambda(\pmb \omega,\mathbf A)$ is simple, we have that\begin{equation}\label{43}
\Gamma'(\lambda_*)=G'(\lambda_*) \neq 0.
\end{equation}
 By Lemma \ref{l41} $\Lambda\tilde \varphi^{-1}$
 is exactly the uniquely determined analytic solution to the equation
$$ \Gamma(\lambda) =0$$
on $\lambda$ in a neighborhood of $(1/f,q,w,b_{11},b_{12},$ $b_{21},b_{22})$ such that
 $$\Lambda\tilde \varphi^{-1}(1/f,q,w,b_{11},b_{12}, b_{21}, b_{22})=\lambda_*.$$ Hence, $\Lambda$  is analytic at $(\pmb \omega,\mathbf A)$. If we replace $(\pmb \omega,\mathbf A)$ by $(\pmb \sigma,\mathbf B)\in\mathcal{F}$,
a similar argument above  yields that $\Lambda$ is analytic at $(\pmb \sigma,\mathbf B)$.
Therefore, $\Lambda$  is analytic in the neighborhood $\mathcal{F}$ of $(\pmb \omega,\mathbf A)$.

Fix a discrete SLE. For
any $(H\,|\,L) \in \rm T_{\mathbf A} \mathcal A^{\mathbb C}$, applying \eqref{41} to $\Gamma$ in a neighborhood of $(\lambda_*,b_{11},b_{12},b_{21},b_{22})$, one can deduce
$$ G'(\lambda_*)\,d(\Lambda\tilde \varphi^{-1})
+\sum^2_{j,k=1}({\partial\Gamma}/{\partial b_{jk}})\,db_{jk}=0,
$$
where$$
\left(\frac{\partial\Gamma}{\partial b_{jk}}\right)_{2\times2}=\begin{pmatrix} b_{22} & -b_{21} \\ -b_{12} & b_{11}
\end{pmatrix}
  +\begin{pmatrix} 1 & 0 \\ 0 & 1
\end{pmatrix}
   \begin{pmatrix} \phi_N(\lambda_*) & f_N\Delta\phi_N(\lambda_*)
         \\ \psi_N(\lambda_*) & f_N\Delta\psi_N(\lambda_*) \end{pmatrix},$$
which, together with \eqref{43}, implies \eqref{42}.

If $\mathbf A\in \mathcal N_{i,j}^{\mathbb{C}}$, where $1\leq i \leq3,\, 3\leq j\leq4$, $i<j$,
the proof can be completed by a method analogous to that used above. \end{proof}

With the help of Lemma \ref{l42}, one can deduce the following result using the same method above:

\begin{thm}\label{th42} Let $\lambda_* \in \mathbb R$ be a simple eigenvalue of a self-adjoint discrete SLP $(\pmb \omega,\mathbf A)\in \Omega_N^{\mathbb R,+}\times\mathcal B^{\mathbb C}$.
Then, the continuous simple eigenvalue branch $\Lambda$ defined on a neighborhood of $(\pmb \omega,\mathbf A)$ in $\Omega_N^{\mathbb R,+}\times\mathcal B^{\mathbb C}$
 through $\lambda_*$ is a $C^{\infty}$ function.\end{thm}

\begin{rem}\label{r41} If $\lambda_* \in \mathbb R$ is an eigenvalue of $(\pmb \omega,\mathbf A)\in \Omega_N^{\mathbb R,+}\times\mathcal B^{\mathbb C}$ with multiplicity 2, the continuous eigenvalue branch $\Lambda$ defined on a neighborhood of $(\pmb \omega,\mathbf A)$ in $\Omega_N^{\mathbb R,+}\times\mathcal B^{\mathbb C}$
 through $\lambda_*$ is not necessarily differentiable. Please see Examples \ref{e54}-\ref{e57} for illustration.\end{rem}

\subsection{ Monotonicity on boundary conditions of continuous eigenvalue branches of self-adjoint discrete SLPs}

In this subsection, we shall investigate monotonicity of continuous simple eigenvalue branches on boundary conditions in several subsets of $\mathcal B^{\mathbb C}$ for  self-adjoint discrete SLPs using the derivative formulas of  continuous simple eigenvalue
branches with respect to the corresponding BC.

\begin{lem}\label{l43} Let $\lambda_*$ be a simple eigenvalue
of a discrete SLP $(\pmb\omega,\mathbf
A) \in \Omega_N^{\mathbb C} \times \mathcal A^{\mathbb C}$, and $\Lambda$ the continuous simple
eigenvalue branch defined on a neighborhood $\mathcal{U}$ of $(\pmb\omega,\mathbf
A)$ in $\Omega_N^{\mathbb C} \times \mathcal A^{\mathbb C}$
through $\lambda_*$. Then, there is a continuous choice
$u(\pmb\sigma,\mathbf B) \in l[0,N+1]$ of
eigenfunction corresponding to $\Lambda(\pmb\sigma,\mathbf B)$ for all
$(\pmb\sigma,\mathbf B) \in \Omega_N^{\mathbb C} \times \mathcal A^{\mathbb C}$
sufficiently close to $(\pmb\omega,\mathbf A)$. Here, the continuity of
$u(\pmb\sigma,\mathbf B)$ means that for each
$(\pmb\tau,\mathbf C)\in \Omega_N^{\mathbb C} \times \mathcal A^{\mathbb C}$ sufficiently close to $(\pmb\omega,\mathbf A)$,
$$
u(\pmb\sigma,\mathbf B) \to u(\pmb\tau,\mathbf C)\;\;\; in\;\;\;
\mathbb C^{N+2} $$
as $(\pmb\sigma,\mathbf B) \to (\pmb\tau,\mathbf C)$ in
$\Omega_N^{\mathbb C} \times \mathcal A^{\mathbb C}$.\end{lem}

\begin{proof} Every eigenfunction of the SLP $(\pmb\omega,\mathbf A)$ corresponding to $\Lambda(\pmb\omega,\mathbf A)=\lambda_*$ can be written as
\begin{equation}\label{44}
u(\pmb\omega,\mathbf A)=C_1(\pmb\omega,\mathbf A)\phi(\Lambda(\pmb\omega,\mathbf A))+C_2(\pmb\omega,\mathbf A)\psi(\Lambda(\pmb\omega,\mathbf A)),
\end{equation}
where $C_1(\pmb\omega,\mathbf A)$, $C_2(\pmb\omega,\mathbf A)\in \mathbb C$ are dependent on $(\pmb\omega,\mathbf A)$. Inserting \eqref{44} into \eqref{12}, we get
\begin{equation}\label{45}
(A+B\Phi_N(\Lambda(\pmb\omega,\mathbf A)))\begin{pmatrix} C_1(\pmb\omega,\mathbf A)\\C_2(\pmb\omega,\mathbf A)\end{pmatrix}=0.
\end{equation}
Set
\begin{equation}\label{46}
M(\pmb\omega,\mathbf A)=(m_{ij}(\pmb\omega,\mathbf A))_{2\times2}:=A+B\Phi_N(\Lambda(\pmb\omega,\mathbf A)).
\end{equation}
Since $\lambda_*$ is simple, $\Lambda(\pmb\sigma,\mathbf B)$ is continuous in $\mathcal{U}$ and has geometric multiplicity 1 for each $(\pmb\sigma,\mathbf B)\in \mathcal U$
 by Theorem \ref{th35} and Corollary \ref{c34}, and one has that\begin{equation}\label{47}
\text{ rank} M(\pmb\omega,\mathbf A)=1,
\end{equation}
which implies $ m_{i_0j_0}(\pmb\omega,\mathbf A)\neq0$ for some $1\leq i_0,j_0\leq2$. Without loss of generality, we assume that $ m_{11}(\pmb\omega,\mathbf A)\neq0$.
 By replacing $(\pmb\omega,\mathbf A)$  with $(\pmb\sigma,\mathbf B)\in \mathcal U$ in \eqref{44}-\eqref{47}, the similar equations denoted by $(4.4')-(4.7')$ still hold. Obviously, there exists a neighborhood $ \mathcal {\hat U}$ of $(\pmb\omega,\mathbf A)$ with $ \mathcal {\hat U}\subset\mathcal U$ such that $m_{11}(\pmb\sigma,\mathbf B)\neq0$ for each $(\pmb\sigma,\mathbf B)\in \mathcal {\hat U}$.
It is evident that
$$
C_{1}(\pmb\sigma,\mathbf B)=m_{12}(\pmb\sigma,\mathbf B),\;\;\;\;\;\;C_{2}(\pmb\sigma,\mathbf B)=-m_{11}(\pmb\sigma,\mathbf B),$$
is a solution of $(4.5')$ for each $(\pmb\sigma,\mathbf B)\in \mathcal {\hat U}$ by $(4.7')$. Hence,
$u(\pmb\sigma,\mathbf B)$ defined by $(4.4')$ is an eigenfunction corresponding to $\Lambda(\pmb\sigma,\mathbf B)$ and continuous in $\mathcal{\hat U}$ by the
fact that $M(\pmb\sigma,\mathbf B)$ is continuous in $\mathcal{U}$. This completes the proof.\end{proof}

\begin{rem}\label{r42} If $\Omega_N^{\mathbb C} \times \mathcal A^{\mathbb C}$ is replaced by $\Omega_N^{\mathbb R,+} \times \mathcal B^{\mathbb C}$ in Lemma \ref{l43},
then it still holds for the self-adjoint discrete SLPs. \end{rem}

By Corollary \ref{c31}, all eigenvalues of every self-adjoint SLP with separated BC are simple. With the help of the preceding lemma we can now give related derivative formulas of continuous simple eigenvalue
branch $\Lambda$ with respect to the parameters of the separated self-adjoint BCs. To indicate the dependence of $\Lambda(\mathbf S_{\alpha,\beta})$ on the parameters  $\alpha$ and $\beta$, we sometimes write $\Lambda(\alpha, \beta)=\Lambda(\mathbf S _{\alpha,\beta})$ for $\mathbf S _{\alpha,\beta} \in \mathcal B_S$.

\begin{thm}\label{th43} Assume that $\lambda_*$ is an
eigenvalue of a self-adjoint discrete SLP
$(\pmb\omega,\mathbf S_{\alpha,\beta})$ with $\mathbf S_{\alpha,\beta}\in\mathcal B_{S}$.
Let $y\in l[0,N+1]$ be a normalized
eigenfunction for $\lambda_*$, and $\Lambda$ the continuous eigenvalue branch
over $\mathcal B_{S}$ through $\lambda_*$.
Then, its derivatives are given by\begin{equation}\label{48}
\Lambda'_{\alpha}(\alpha, \beta)=-|y_0|^2-|f_0\Delta y_0|^2,\;\;
\Lambda'_{\beta}(\alpha, \beta)=|y_N|^2+|f_N\Delta y_N|^2.
\end{equation}
\end{thm}

\begin{proof} We first show that the first relation in \eqref{48} holds.
Fix all the components of $(\pmb\omega,\mathbf S_{\alpha,\beta})$ except $\alpha$.
Let $y=y(\cdot,\alpha)$. By Corollary \ref{c31}, $\lambda_*$ is a simple
eigenvalue of $(\pmb\omega,\mathbf S_{\alpha,\beta})$. By Remark \ref{r42}, we can choose
 an eigenfunction $z=y(\cdot,\alpha+h)$ with respect to $\Lambda(\alpha+h,\beta)$
 for $h\in \mathbb R$
sufficiently small such that $z\to y$ as $h\to 0$. From \eqref{11} we get that
$$\left(\Lambda(\alpha+h,\beta)-\Lambda(\alpha,\beta)\right)
w_ny_n\bar{z}_n=-\Delta [y_{n-1},z_{n-1}],$$
where $[y_{n},z_{n}]:=y_n(f_n\Delta\bar{z}_n)-(f_n\Delta y_n)\bar{z}_n$.
Hence,\begin{equation}\label{49}
\left(\Lambda(\alpha+h,\beta)-\Lambda(\alpha,\beta)\right)\sum\limits_{n=1}^Nw_ny_n\bar{z}_n
 =[y_0,z_0]-[y_N,z_N].
\end{equation}
The BC $\mathbf S_{\alpha,\beta}$ with respect to $\beta$ implies that
\begin{equation}\label{410}
[y_N,z_N]=0.
\end{equation}
In the case that $\alpha\neq \pi/2$, by the BC $\mathbf S_{\alpha,\beta}$
with respect to $\alpha$,
together with \eqref{49} and \eqref{410}, we get that
\begin{equation}\label{411}
 \left(\Lambda(\alpha+h,\beta)-\Lambda(\alpha,\beta)\right)\sum\limits_{n=1}^N
 w_ny_n\bar{z}_n
 =-(\tan (\alpha+h)-\tan\alpha)(f_0\Delta y_0)(f_0\Delta \bar{z}_0).
\end{equation}
 Dividing both sides of \eqref{411} by $h$ and taking the limit as $h\to 0$, we obtain that
$$
\Lambda'_{\alpha}(\alpha,\beta)=-|f_0\Delta y_0|^2\sec^2\alpha =-|y_0|^2-|f_0\Delta y_0|^2.
$$
In the other case that $\alpha= \pi/2$,
by the BC $\mathbf S_{\pi/2,\beta}$ with respect to $\alpha$, together with \eqref{49} and \eqref{410},
we get that $f_0\Delta y_0=0$ and
\begin{equation}\label{412}
\left(\Lambda(\pi/2+h,\beta)-\Lambda(\pi/2,\beta)\right)\sum\limits_{n=1}^Nw_ny_n\bar{z}_n
 =\cot(\pi/2+h)y_0\bar{z}_0.
\end{equation}
Dividing both sides of \eqref{412} by $h$ and taking the limit as $h\to 0$, we obtain that
$$
\Lambda'_{\alpha}(\pi/2,\beta)=-|y_0|^2.
$$
Hence, the first relation in \eqref{48} follows.

With a similar argument to that used in the above discussion, one can show that the second relation in \eqref{48} holds. This completes the proof.\end{proof}

The following result is directly derived from Corollary \ref{c31} and Theorem \ref{th43}.

\begin{thm}\label{th44} Assume that \eqref{11} is in
$\Omega_N^{\mathbb R,+}$. Then each continuous eigenvalue branch over
$\mathcal B_{S}$ is always
strictly decreasing in the $\alpha$-direction and always
strictly increasing in the $\beta$-direction.\end{thm}

The following results give the derivative formulas of continuous simple eigenvalue
branch $\Lambda$ over $\Omega_N^{\mathbb R,+} \times\mathcal B^{\mathbb C}$ with respect to $\mathcal B^{\mathbb C}$.

\begin{thm}\label{th45} Fix
$\pmb\omega\in\Omega_N^{\mathbb R,+}$. Let $\lambda_*$ be a simple eigenvalue of
$(\pmb\omega,\mathbf{A})$ for a
self-adjoint boundary condition $\mathbf A$, $y \in l[0,N+1]$ be a normalized
eigenfunction for $\lambda_*$, and $\Lambda$ be the continuous simple
eigenvalue branch of $(\pmb \omega,\mathbf{B})$ for $\mathbf{B} \in\mathcal B^{\mathbb C}$
through $\lambda_*$. Then,
we have the following derivative formulas:
\begin{itemize}
 \item[{\rm (1)}]  when $\mathbf A\in\mathcal O_{1,3}^{\mathbb C}$,
 \begin{equation}\label{413}
d\Lambda\big|_{\mathbf A}(H\,|\,L) =\left(f_0 \Delta\bar{y}_0, f_N \Delta\bar{y}_N\right)
\begin{pmatrix} h_{12}&\bar{h}_{22} \\ h_{22}&\l_{22} \end{pmatrix}
\begin{pmatrix}f_0 \Delta y_0 \\ f_N \Delta y_N \end{pmatrix}
\end{equation}
for all $(H\,|\,L)$ in$$
\mathrm {T}_{\mathbf A}\mathcal B^{\mathbb C}=\mathrm {T}_{\mathbf A}\mathcal O_{1,3}^{\mathbb C}
=\left\{ \begin{pmatrix} 0&h_{12}&0&\bar{h}_{22} \\
0&h_{22}&0&l_{22}
\end{pmatrix}:
        h_{12},l_{22} \in \mathbb R, \ h_{22} \in \mathbb C  \right\};
$$

 \item[{\rm (2)}]  when $\mathbf A\in\mathcal O_{1,4}^{\mathbb C}$,
 \begin{equation}\label{414}
d\Lambda\big|_{\mathbf A}(H\,|\,L) =\left(f_0 \Delta\bar{y}_0,\bar{y}_N \right)
\begin{pmatrix} h_{12}&\bar{h}_{22} \\ h_{22}&l_{21}
\end{pmatrix}
\begin{pmatrix} f_0 \Delta y_0 \\ y_N \end{pmatrix}
\end{equation}
 for all $(H\,|\,L)$ in
$$
\mathrm {T}_{\mathbf A}\mathcal B^{\mathbb C}=\mathrm {T}_{\mathbf A}\mathcal O_{1,4}^{\mathbb C}
=\left\{ \begin{pmatrix} 0&h_{12}&\bar{h}_{22}&0 \\
0&h_{22}&l_{21}&0
\end{pmatrix}:
         h_{12},l_{21} \in \mathbb R, \ h_{22} \in \mathbb C \right\};
$$

 \item[{\rm (3)}]  when $\mathbf A\in\mathcal O_{2,3}^{\mathbb C}$,\begin{equation}\label{415}
d\Lambda\big|_{\mathbf A}(H\,|\,L) =\left(\bar{y}_0, f_N \Delta\bar{y}_N \right)
\begin{pmatrix} h_{11}&\bar{h}_{21} \\ h_{21}&l_{22}
\end{pmatrix}
\begin{pmatrix} y_0 \\ f_N \Delta y_N  \end{pmatrix}
\end{equation}
for all $(H\,|\,L)$ in
$$
\mathrm {T}_{\mathbf A}\mathcal B^{\mathbb C}=\mathrm {T}_{\mathbf A}\mathcal O_{2,3}^{\mathbb C}
=\left\{ \begin{pmatrix} h_{11}&0&0&\bar{h}_{21} \\
h_{21}&0&0&l_{22}
\end{pmatrix}:
          h_{11},l_{22} \in \mathbb R, \ h_{21} \in \mathbb C \right\};
$$

 \item[{\rm (4)}] when $\mathbf A\in\mathcal O_{2,4}^{\mathbb C}$,\begin{equation}\label{416}
 d\Lambda\big|_{\mathbf A}(H\,|\,L) =\left(\bar{y}_0, \bar{y}_N\right)
\begin{pmatrix} h_{11}&\bar{h}_{21} \\ h_{21}&l_{21} \end{pmatrix}
\begin{pmatrix} y_0 \\ y_N \end{pmatrix}
\end{equation}
for all $(H\,|\,L)$ in
$$
\mathrm {T}_{\mathbf A}\mathcal B^{\mathbb C} =\mathrm {T}_{\mathbf A}\mathcal O_{2,4}^{\mathbb C}
=\left\{ \begin{pmatrix} h_{11}&0&\bar{h}_{21}&0 \\
h_{21}&0&l_{21}&0
\end{pmatrix}:
         h_{11},l_{21} \in \mathbb R, \ h_{21} \in \mathbb C \right\}.
$$
\end{itemize}
\end{thm}
\begin{proof} We first show that \eqref{413} holds.  For
$\mathbf A=[A\,|\,B]\in\mathcal O_{1,3}^{\mathbb C}$ given by \eqref{26},
we define
$$\mathbf{B}:=[ A+H\ |\ B+L ]=\left [\begin{array} {cccc}1&a_{12}+h_{12}&0&\bar{z}+\bar{h}_{22}\\
0&z+h_{22}&-1&b_{22}+l_{22}\end{array}  \right ],$$
where
$h_{12}$, $l_{22}\in \mathbb R$, $h_{22}\in\mathbb C$. Obviously, $\mathbf{B}\rightarrow\mathbf{A}$ as $(h_{12},h_{22},l_{22})\rightarrow0$.
By Remark \ref{r42}, we can choose an eigenfunction $z=z(\cdot,\mathbf B)$ of $\Lambda(\pmb{\omega},\mathbf B)$ such that $z\rightarrow y$ as $\mathbf{B}\rightarrow\mathbf{A}$.
 From \eqref{11}, we get $$
\left(\Lambda(\mathbf{B})-\Lambda(\mathbf{A})\right)
\sum\limits_{n=1}^Nw_ny_n\bar{z}_n=[y_0,z_0]-[y_N,z_N],
$$
which, together with the boundary conditions,
$$
A\begin{pmatrix}y_0 \\ f_0\Delta y_0\end{pmatrix}+B\begin{pmatrix}y_N \\ f_N\Delta y_N\end{pmatrix}=0,
$$
and
\vspace{-0.05cm}$$
(A+H)\begin{pmatrix}z_0 \\ f_0\Delta z_0\end{pmatrix}+(B+L)\begin{pmatrix}z_N \\ f_N\Delta z_N\end{pmatrix}=0,
$$
implies that
$$
\begin{array}{ll}&\left(\Lambda(\mathbf{B})-\Lambda (\mathbf{A})\right)\sum\limits_{n=1}^Nw_ny_n\bar{z}_n\\
=&(f_0\Delta y_0)(f_0\Delta \bar{z}_0)h_{12}
+(f_N\Delta y_N)(f_0\Delta \bar{z}_0)\bar{h}_{22}\vspace{0.2cm}\\
&+(f_0\Delta y_0)(f_N\Delta \bar{z}_N)h_{22}+(f_N\Delta y_N)(f_N\Delta\bar{z}_N)l_{22}.
\end{array}$$
Further, using the following equalities
$$
{\partial}/{\partial z}=({1}/{2})({\partial}/{\partial z_1}-i{\partial}/{\partial z_2}),\;\;\;\;
{\partial}/{\partial \bar z}={1}/{2}({\partial}/{\partial z_1}+i{\partial}/{\partial z_2}),
$$
where $z=z_1+i z_2$ with $z_1,z_2\in \mathbb{R}$,
one can easily conclude that \eqref{413} holds. With similar arguments, one can show that
\eqref{414}, \eqref{415}, and \eqref{416} hold. This proof is complete. \end{proof}

Next, we give an important application of Theorem \ref{th45}.

\begin{thm}\label{th46} Assume that \eqref{11} is in
$\Omega_N^{\mathbb R,+}$. Then, in each of the coordinate systems
$\mathcal O_{1,3}^{\mathbb C}$, $\mathcal O_{1,4}^{\mathbb C}$, $\mathcal
O_{2,3}^{\mathbb C}$, and $\mathcal O_{2,4}^{\mathbb C}$ in $\mathcal B^{\mathbb C}$,
every continuous eigenvalue branch is always increasing in the two
real axis directions.\end{thm}

For example, in $\mathcal O_{1,3}^{\mathbb C}$, every continuous eigenvalue
branch is always increasing in the $a_{12}$-direction and in the
$b_{22}$-direction. Note that the monotonicity in Theorem \ref{th46} is not necessarily strict
(see Example \ref{e54}).

\begin{proof} Let $z\in \mathbb C$ and $b_{22}\in \mathbb R$. By Lemma 3.7 in \cite{Peng} we know
that$$
C_{z,b_{22}}:=\left\{\left[\begin{array}{cccc} 1&s&0&\bar {z}\\
0&z&-1&b_{22}\end{array}\right]:\;s\in\mathbb R\right\}\bigcup\left\{\left[\begin{array}{cccc} 0&1&0&0\\
0&0&-1&b_{22}\end{array}\right]\right\}
$$
is a real-analytic loop. Let $\Lambda$ be a
continuous eigenvalue branch on a subset of $C_{z,b_{22}}$. Note that both $ C_{z,b_{22}}$ and the curve
 $$ \lambda\mapsto
[\Phi _N(\lambda)\,|\,-I],\;\;\;\;\lambda\in \mathbb R,$$
are real-analytic. So, either their intersection is discrete in $C_{z,b_{22}}$ or they
agree completely.

In the former case, $\Lambda$ is simple on a dense subset of its domain.
Fix an $s_0\in\mathbb R$ and $\delta >0$ such that
$$
\mathbf A(s):=\left[\begin{array}{cccc} 1&s&0&\bar{z}\\
0&z&-1&b_{22}\end{array}\right],\;\;s\in(s_0-\delta,s_0+\delta),
$$
lies in the domain of $\Lambda$ when $\Lambda (\mathbf A(s_0))$
is simple. Assume that $y$ is a normalized
eigenfunction for $\Lambda (\mathbf A(s_0))$. Then, by \eqref{413} we get that
$$
d\Lambda\big|_{\mathbf A}((H\,|\,L)) =|f_0 \Delta y_0|^2\geq 0.
$$
Since $\Lambda$ is simple in a dense open subset of $(s_0-\delta,s_0+\delta)$,
a similar argument implies that $\Lambda$ has a non-negative derivative
at each point in the dense open subset. Thus, $\Lambda$ is increasing
in the dense open subset. Assume that $\Lambda (\mathbf A(s_1))$ is an eigenvalue of
multiplicity 2 with $s_1\in (s_0-\delta,s_0+\delta)$. By the continuity of $\Lambda$
one has that
$$\lim_{s\rightarrow {s_1}^-}\Lambda (\mathbf A(s))
=\Lambda (\mathbf A(s_1))=\lim_{s\rightarrow {s_1}^+}\Lambda (\mathbf A(s)),
$$
which, together with the monotonicity of $\Lambda$
in the dense open subset, implies that $\Lambda$ is increasing in a neighborhood of $s_1$.
Thus, $\Lambda$ is increasing in
$(s_0-\delta,s_0+\delta)$.

The latter case may happen for at most one pair $z\in\mathbb R$
and $b_{22}\in\mathbb R$. The monotonicity of $\Lambda$ can be
deduced from the former case by perturbing $z$ or $b_{22}$.
For example, by perturbing $b_{22}$, namely by $b'_{22}$,
$\Lambda$ is increasing on the domain according to discussion in the former case. So, setting
that $s_1<s_2$, we have that
$$\Lambda (\mathbf A(s_1,b_{22}))=\lim_{b'_{22}\rightarrow b_{22}}\Lambda (\mathbf A(s_1,b'_{22}))\leq\lim_{b'_{22}\rightarrow b_{22}}
\Lambda (\mathbf A(s_2,b'_{22}))=\Lambda (\mathbf A(s_2,b_{22})).$$

 One can show the rest of the claims similarly.
The proof is complete.\end{proof}

\subsection{ Monotonicity on Sturm-Liouville equations of continuous eigenvalue branches of self-adjoint discrete SLPs}

\begin{lem}\label{l44} If $u$
and $v$ are eigenfunctions for eigenvalues of two self-adjoint
discrete SLPs $((1/f,q,w),\mathbf A)$ and
$((1/g,r,s),\mathbf A)$, respectively, with the same
BC $\mathbf A$, then\begin{equation}\label{417}
u_0 ({g_0\Delta\bar{v}_{0}}) -(f_0\Delta u_0)
\bar{v}_{0} =u_N  ({g_N\Delta\bar{v}_{N}}) -(f_N\Delta
u_N)  \bar{v}_{N}.
\end{equation}
\end{lem}

\begin{proof} First, consider the coupled self-adjoint BC $\mathbf{A}$.
It follows from Lemma \ref{l21} that $\mathbf A
=[e^{i\gamma}K\,|\,-I]$ for some $\gamma \in [0,\pi)$ and $K \in \rm
SL(2,\mathbb R)$. So, we have that
$$
K^{\rm t} E K =E,
$$
\vspace{-0.05cm}$$
\begin{pmatrix} u_N \\ f_N\Delta u_N \end{pmatrix} = e^{i\gamma} K
\begin{pmatrix} u_0 \\ f_0\Delta u_0 \end{pmatrix}, \;\;\;\; \begin{pmatrix} v_N \\
g_N\Delta v_N \end{pmatrix} = e^{i\gamma} K \begin{pmatrix} v_0 \\
g_0\Delta v_0 \end{pmatrix}.
$$
Thus,
$$\begin{array}{llll}
 &
  ({g_N\Delta \bar{v}_N}) u_N
  -  \bar{v}_N(f_N\Delta u_N)
=\begin{pmatrix} v_N \\ g_N\Delta v_N \end{pmatrix}^{*}
   E \begin{pmatrix} u_N \\ f_N\Delta u_N \end{pmatrix}\\ \vspace{0.2cm}
  =&\begin{pmatrix} v_0 \\ g_0\Delta v_0 \end{pmatrix}^{*}
   K^{\rm t} E K \begin{pmatrix} u_0 \\ f_0\Delta u_0 \end{pmatrix}
  =\begin{pmatrix} v_0 \\ g_0\Delta v_0 \end{pmatrix}^{*}
   E \begin{pmatrix} u_0 \\ f_0\Delta u_0 \end{pmatrix}\\ \vspace{0.2cm}
      = &({g_0\Delta \bar{v}_0})u_0
   - \bar{v}_0(f_0\Delta u_0) .
\end{array}$$
Hence, \eqref{417} holds in the coupled case.

The separated case can be treated similarly. The proof is complete.\end{proof}

\begin{thm}\label{th47} Fix $\mathbf{A}\in \mathcal{B}^{\mathbb{C}}$. Let $\lambda_*$ be a simple
eigenvalue of a self-adjoint discrete SLP
$(\pmb\omega,\mathbf A)=((1/f,q,w),\mathbf A)$, $y\in l[0,N+1]$ a normalized
eigenfunction for $\lambda_*$, and $\Lambda$ the continuous
simple eigenvalue branch over $\Omega_N^{\mathbb R,+}$ through $\lambda_*$.
Then,
\begin{equation}\label{418}
d\Lambda\big|_{\pmb\omega}(h,k,l) =-\sum_{n=0}^{N-1}
|f_n\Delta y_n|^2 h_n
 +\sum_{n=1}^N |y_n|^2 k_n
 -\lambda_* \sum_{n=1}^N |y_n|^2 l_n
\end{equation}
for all
$(h,k,l) =\big( (h_0,\cdots,h_N), (k_1,\cdots,k_N),
(l_1,\cdots,l_N) \big)\in \rm T_{\pmb\omega}\Omega_N^{\mathbb R,+}=\mathbb R^{N+1}
\times \mathbb R^N \times \mathbb R^N.$\end{thm}

\begin{proof} Denote $\Lambda=\Lambda(1/f,q,w)$ and $y=y(\cdot,1/f,q,w)$.
By Remark \ref{r42} we can choose an eigenfunction $z=y(\cdot,1/f+h,q+k,w+l)$
with respect to $\lambda(1/f+h,q+k,w+l)$ for $(h,k,l)\in\mathbb R^{N+1} \times \mathbb R^N \times \mathbb R^N$
sufficiently small such that $z\to y$ as  $(h,k,l)\to0$. For convenience, we set $1/g=1/f +h$ with $g=\{g_{n}\}_{n=0}^{N}$, $\hat q=q+k$, $\hat w=w+l$.
Using \eqref{11} and Lemma
\ref{l44}, we get that\vspace{-0.1cm}
\begin{align*}\vspace{-0.5cm}
&\left(\Lambda(1/g,\hat q,\hat w)-\Lambda(1/f,q,w)\right)\sum\limits_{n=1}^Nw_ny_n\bar{z}_n
\\\vspace{-0.5cm}&=\sum\limits_{n=1}^N \left(\bar{z}_n\nabla (f_n\Delta
y_n)- y_n\nabla (g_n\Delta
\bar{z}_n)\right)-\Lambda(1/g,\hat q,\hat w)\sum\limits_{n=1}^Nl_ny_n\bar{z}_n
+\sum\limits_{n=1}^Nk_ny_n\bar{z}_n\\\vspace{-0.5cm}
&=\sum\limits_{n=1}^N \Delta y_{n-1}(g_{n-1}\Delta
\bar{z}_{n-1})-\sum\limits_{n=1}^N \Delta \bar{z}_{n-1}(f_{n-1}\Delta
y_{n-1})+y_0(g_0\Delta \bar{z}_0)\vspace{-0.5cm}
\end{align*}
\vspace{-0.1cm}$$-y_N(g_N\Delta
\bar{z}_N)
+\bar{z}_N(f_N\Delta y_N)-\bar{z}_0(f_0\Delta y_0)-\Lambda(1/g,\hat q,\hat w)\sum\limits_{n=1}^Nl_ny_n\bar{z}_n+\sum\limits_{n=1}^Nk_ny_n\bar{z}_n$$\vspace{-0.1cm}
\begin{align*}\vspace{-0.5cm}
&=\sum\limits_{n=0}^{N-1} (f_{n}\Delta y_{n})(g_{n}\Delta
\bar{z}_{n})(1/f_{n}-1/g_{n})-\Lambda(1/g,\hat q,\hat w)\sum\limits_{n=1}^Nl_ny_n\bar{z}_n
+\sum\limits_{n=1}^Nk_ny_n\bar{z}_n
\\\vspace{-0.5cm} &=-\sum\limits_{n=0}^{N-1}
(f_{n}\Delta y_{n})(g_{n}\Delta
\bar{z}_{n})h_{n}-\Lambda(1/g,\hat q,\hat w)\sum\limits_{n=1}^{N}l_ny_n\bar{z}_n
+\sum\limits_{n=1}^{N}k_ny_n\bar{z}_n,
\end{align*}
which yields that \eqref{418} holds. This completes the proof.\end{proof}

\begin{thm}\label{th48} Fix a self-adjoint BC. Then,
each continuous eigenvalue branch $\Lambda$ over
$\Omega_N^{\mathbb R,+}$ is decreasing in every $(1/f_n)$-direction with
$0 \leq n \leq N-1$, independent of $f_N$, and increasing in every
$q_n$-direction; while the positive parts of $\Lambda$ are
decreasing in every $w_n$-direction, and the negative parts of
$\Lambda$ are increasing in every $w_n$-direction.\end{thm}

\begin{proof} The proof is similar to that of Theorem \ref{th46} and
hence omitted.\end{proof}

\begin{rem}\label{r43} The monotonicity in Theorem \ref{th48} is not necessarily strict.
Please see Examples \ref{e55}-\ref{e57} for illustration.\end{rem}

\section{ Examples}

In this section, we shall give some examples to illustrate some results obtained in
Sections 3 and 4.

Consider the modified discrete Fourier equation,
i.e., the discrete SLE \eqref{11} with
$$
N=2, \; f_0 =f_1 =1, \; q_1 =q_2 =0, \; w_1 =1.
$$
From \eqref{31} and \eqref{32} we deduce that
\begin{equation}\label{51}
\Phi_0(\lambda) =\begin{pmatrix} 1 & 0\\0 & 1
\end{pmatrix},
\Phi_1(\lambda) =\begin{pmatrix} 1 & 1 \\ -\lambda & 1-\lambda
\end{pmatrix},
$$
$$
\Phi_2(\lambda) = \begin{pmatrix}1-\lambda &
2-\lambda\\ -(1+w_2)\lambda+w_2\lambda^2 & 1-(1+2w_2)\lambda+w_2\lambda^2
\end{pmatrix}.
\end{equation}
Further, when $\omega _2=1$, we have that
\begin{equation}\label{52}
\Phi_2(\lambda)
=\begin{pmatrix} 1-\lambda & 2-\lambda
       \\ -2\lambda+\lambda^2 & 1-3\lambda+\lambda^2 \end{pmatrix}.
\end{equation}

We first give two examples to show that the analytic and geometric multiplicities of an eigenvalue are not necessarily equal for a discrete SLP, which
is not self-adjoint.

\begin{ex}\label{e51}
 Consider the  modified discrete Fourier equation with $\omega_2=1$, and let $c \in \mathbb C$. Then, by using Lemma \ref{l33} and \eqref{52},  the characteristic function for the separated BC
$$
\mathbf A(c): =\left[\begin{array}{cccc} c & 2c+1 & 0 & 0  \\ 0 & 0 & c & 1 \end{array}\right]
$$
is
$$
\Gamma(\lambda) =(c^2+2c+2) \lambda -(c+1) \lambda^2.
$$
Thus, 0 is an eigenvalue for $\mathbf A(-1\pm i)$ with geometric
multiplicity 1 and analytic multiplicity 2. So,  the
analytic and geometric multiplicities of an eigenvalue are not
equal in general. Note that the BCs $\mathbf A(-1\pm i)$ are not
self-adjoint.\end{ex}

\begin{ex}\label{e52} Take the discrete SLP consisting of the modified
discrete Fourier equation and the separated BC
$$
\left[\begin{array}{cccc} \cos\alpha & 2\cos\alpha-\sin\alpha &          0 &           0 \\
                  0 &                      0 & \cos\alpha & -\sin\alpha \end{array}\right],
$$
where $\alpha \in [0,\pi)$. By using Lemma
\ref{l33} and \eqref{51}, direct calculations yield that the characteristic function of
the problem is
$$
\Gamma(\lambda) =(1-\sin(2\alpha)+w_2\sin^2\alpha) \lambda
                    +(\cos\alpha-\sin\alpha) w_2 \sin\alpha \cdot \lambda^2.
$$
Thus, if
\begin{equation}\label{53}
\alpha \in (0,\pi/4) \cup (\pi/4,\pi), \; w_2 =(\sin(2\alpha)-1)
/\sin^2\alpha,
\end{equation}
then 0 is an eigenvalue with geometric multiplicity 1 and analytic
multiplicity 2. In this case, the weight function $w$ is indefinite
since $w_2 <0$ by \eqref{53}.
\end{ex}

The following examples below show us continuous eigenvalue branches in different cases.

\begin{ex}\label{e53} Consider the discrete Fourier equation with $\omega_2=1$,
and let $\alpha \in \mathbb R$. Then, from Lemma \ref{l33}
and \eqref{52}, we see that the characteristic function for the BC
$\mathbf S_{\alpha,\pi}$ (see \eqref{25}) is
$$
\Gamma(\lambda) =-(1-\lambda)\sin\alpha -(2-\lambda) \cos\alpha
=-2\cos\alpha -\sin\alpha +(\cos\alpha+\sin\alpha) \lambda.
$$
Thus, the self-adjoint discrete SLP consisting of the discrete
Fourier equation and $\mathbf S_{3\pi/4,\pi}$ has no
eigenvalues; and when $\alpha \in [0,3\pi/4) \cup (3\pi/4,\pi)$,
the only eigenvalue for $\mathbf S_{\alpha,\pi}$ is
\begin{equation}\label{54}
\lambda_1 =(2\cos\alpha+\sin\alpha)/ (\cos\alpha+\sin\alpha).
\end{equation}
So, in this case, there is only one continuous eigenvalue branch
over $\big\{ \mathbf S_{\alpha,\pi}: \ \alpha \in [0,3\pi/4) \cup (3\pi/4,\pi) \big\}$, i.e., the
function given by \eqref{54} for $\alpha \in [0,3\pi/4) \cup (3\pi/4,\pi)$. See
Figure 5.1.

\begin{center}

\includegraphics[width=75mm]{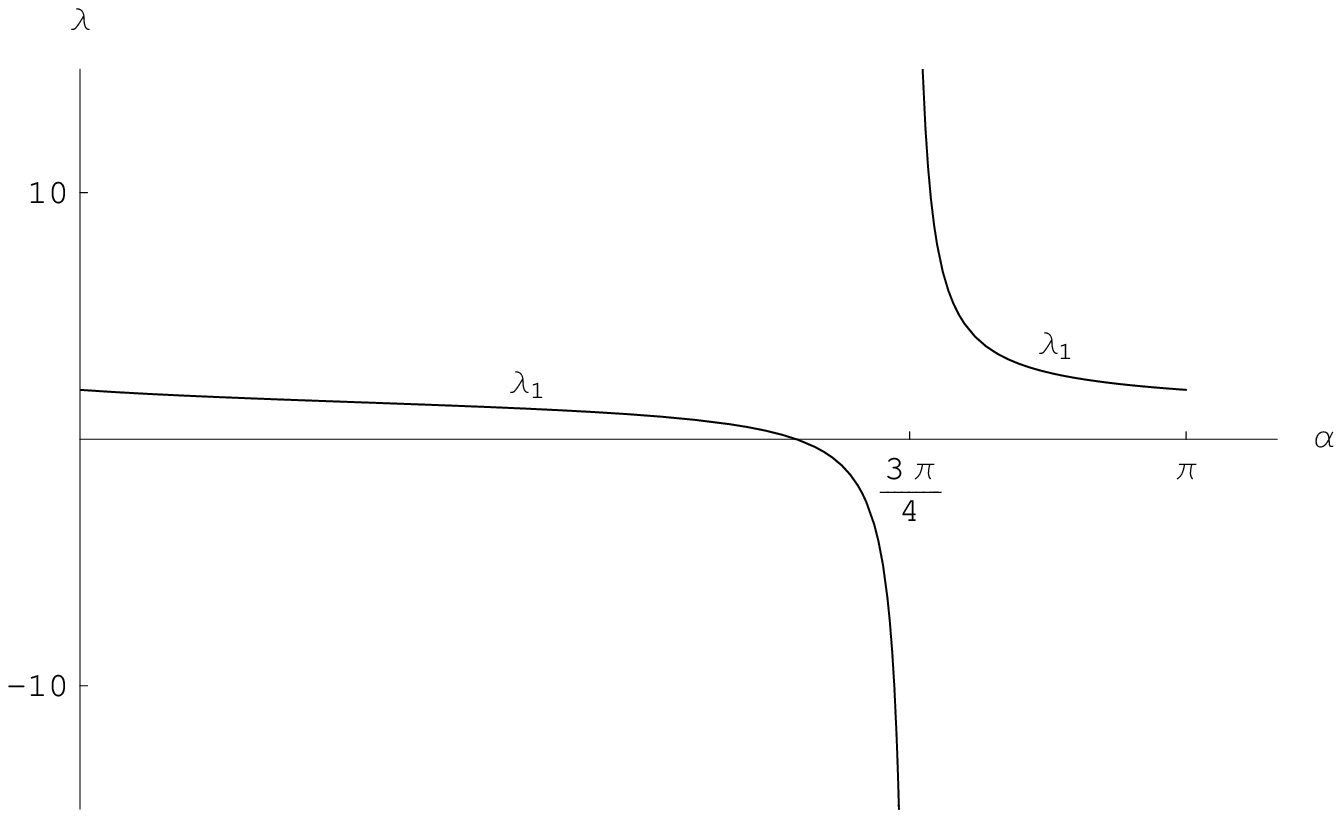}

{\small {\bf Figure 5.1.} Only one continuous eigenvalue branch}

\end{center}
\vspace{0.2cm}
\rm Similarly, the characteristic function for the BC $\mathbf
S_{\alpha,\pi/2}$ is $$
\Gamma(\lambda) =-\cos\alpha +(3\cos\alpha+2\sin\alpha) \lambda
-(\cos\alpha+\sin\alpha) \lambda^2.
$$
Thus, the only eigenvalue for $\mathbf S_{3\pi/4,\pi/2}$ is $\lambda_1
=1$; and when $\alpha \in [0,3\pi/4) \cup (3\pi/4,\pi)$, the two
eigenvalues for $\mathbf S_{\alpha,\pi/2}$ are
$$
\lambda_1(\alpha) =\begin{cases} \lambda_-(\alpha) & \text{ if }
\alpha \in [0,3\pi/4),
     \\ \lambda_+(\alpha) & \text{ if } \alpha \in (3\pi/4,\pi),
     \end{cases}
\qquad \lambda_2(\alpha) =\begin{cases} \lambda_+(\alpha) & \text{
if } \alpha \in [0,3\pi/4),
     \\ \lambda_-(\alpha) & \text{ if } \alpha \in (3\pi/4,\pi),
     \end{cases}
$$
$$
\lambda_{\pm}(\alpha) ={3\cos\alpha+2\sin\alpha \pm
\sqrt{\cos^2\alpha+4\sin(2\alpha)+4} \over
  2(\cos\alpha+\sin\alpha) }.
$$
So, in this case, each continuous eigenvalue branch over $\big\{
\pmb S_{\alpha,\pi/2}: \ \alpha \in [0,\pi) \big\}$ is locally a part of one
of the following functions:
$$\begin{array}{llll}
\Lambda_1(\alpha) & =&
    \lambda_1(\alpha) \text{ \ for } \alpha \in [0,3\pi/4),\\
\Lambda_{2,1}(\alpha) & =& \begin{cases} \lambda_2(\alpha) &
\text{ if } \alpha \in [0,3\pi/4),
    \\ 1 & \text{ if } \alpha =3\pi/4,
    \\ \lambda_1(\alpha) & \text{ if } \alpha \in (3\pi/4,\pi),
    \end{cases}\\
\Lambda_2(\alpha) & = &\lambda_2(\alpha) \text{ \ for } \alpha \in
(3\pi/4,\pi),
\end{array}$$

See Figure 5.2. Note that for each $\mathbf S_{\alpha,\pi/2}$,
there are one or two continuous eigenvalue branches defined on a
neighborhood of $\mathbf S_{\alpha,\pi/2}$ in $\big\{ \mathbf S_{\alpha,\pi/2}:
\ \alpha\in [0,\pi) \big\}$. This example demonstrates that the
index of the eigenvalue in a continuous eigenvalue branch over
$\Omega_N^{\mathbb R,+} \times \mathcal B^{\mathbb C}$ can change as the
problem varies. \begin{center}

\includegraphics[width=75mm]{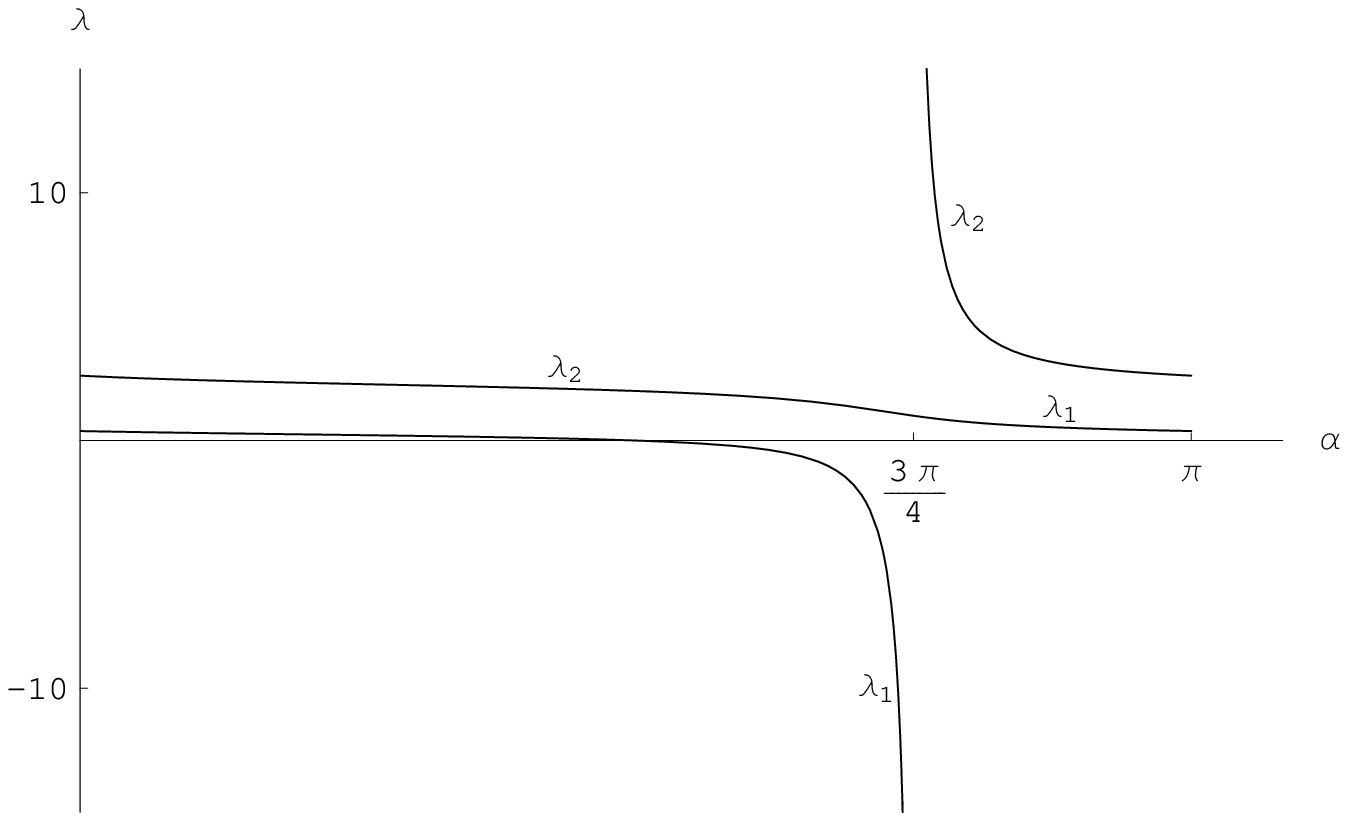}

{\small {\bf Figure 5.2.} One or two continuous eigenvalue branches}
\end{center}
\vspace{0.2cm}

Note that Figures 5.1 and 5.2 also agree with the strict
monotonicity of continuous eigenvalue branches in the
$\alpha$-direction given in Theorem \ref{th44}.\end{ex}

The next example shows that the monotonicity in Theorem \ref{th46} is not necessarily strict.

\begin{ex}\label{e54} Consider the discrete Fourier equation with  $\omega_2=1$.
Let $a_{12} >1$ and $b_{21} \in \mathbb R$. Then, the
characteristic function for the BC
$$
\mathbf A(a_{12},b_{21}): =\left[\begin{array}{cccc} 1 & a_{12} &     -1 & 0 \\
                                 0 &     -1 & b_{21} & 1 \end{array}\right]
\in \mathcal O_{1,4}^{\mathbb C}
$$
is
$$
\Gamma(\lambda) =-(a_{12}-2)b_{21}
+[(a_{12}-1)b_{21}+2(a_{12}-2)] \lambda
                    -(a_{12}-1) \lambda^2.
$$
Thus, the two eigenvalues for $\mathbf A(a_{12},b_{21})$ are
\begin{equation}\label{55}
\lambda_1(a_{12},b_{21}) =
{(a_{12}-1)b_{21}+2(a_{12}-2)-\delta^{1 \over 2}(a_{12},b_{21}) \over
2(a_{12}-1)},
\end{equation}
\begin{equation}\label{56}
\lambda_2(a_{12},b_{21})  =
{(a_{12}-1)b_{21}+2(a_{12}-2)+\delta^{1 \over 2}(a_{12},b_{21}) \over
2(a_{12}-1)},
\end{equation}
where $$
\delta(a_{12},b_{21}) =(a_{12}-1)^2b_{21}^2+4(a_{12}-2)^2.
$$
Let $\Lambda_1(a_{12},b_{21})=\lambda_1(a_{12},b_{21})$ and
 $\Lambda_2(a_{12},b_{21})=\lambda_2(a_{12},b_{21})$.
These are the two continuous eigenvalue branches over
$$
\big\{ \mathbf A(a_{12},b_{21}): \ a_{12}>1, \, b_{21} \in \mathbb R \big\}.
$$

Let $R =(1,+\infty) \times \mathbb R$. Since $\delta(a_{12},b_{21})>0$ for each $(a_{12},b_{21})\in R \setminus \{ (2,0) \}$, then $\lambda_1(a_{12},b_{21})$ and $\lambda_2(a_{12},b_{21})$
are two different and simple eigenvalues in this case. By Lemma \ref{l42} and Theorem \ref{th42}, $\Lambda_1$ and
$\Lambda_2$ are the only two different $C^{\infty}$ eigenvalue branches on
$R \setminus \{ (2,0) \}$.

On the other hand, setting $b_{21}=0$, from \eqref{55} and \eqref{56} we deduce that
$$\begin{array}{ll}\vspace{0.2cm}
\lambda_1(a_{12},0)  =
\begin{cases}                      0 & \text{ if } a_{12}>2, \\
       2(a_{12}-2)/(a_{12}-1) & \text{ if } 1<a_{12}<2, \end{cases}
\\
\lambda_2(a_{12},0)  =
\begin{cases} 2(a_{12}-2)/(a_{12}-1) & \text{ if } a_{12}>2, \\
                            0 & \text{ if } 1<a_{12}<2. \end{cases}
\end{array}$$
See Figure 5.3. Therefore, the two continuous eigenvalue
branches $\lambda_1(a_{12},0)$ and $\lambda_2(a_{12},0)$
are not differentiable at $a_{12} =2$. Note that $\lambda=0$
is the eigenvalue of multiplicity 2 of the problem when $a_{12}=2$ and $b_{21}=0$. This
demonstrates that the multiplicity assumptions in
Theorems \ref{th41} and \ref{th42} can not be omitted in general (see Remark \ref{r41}).
\begin{center}

\includegraphics[width=75mm]{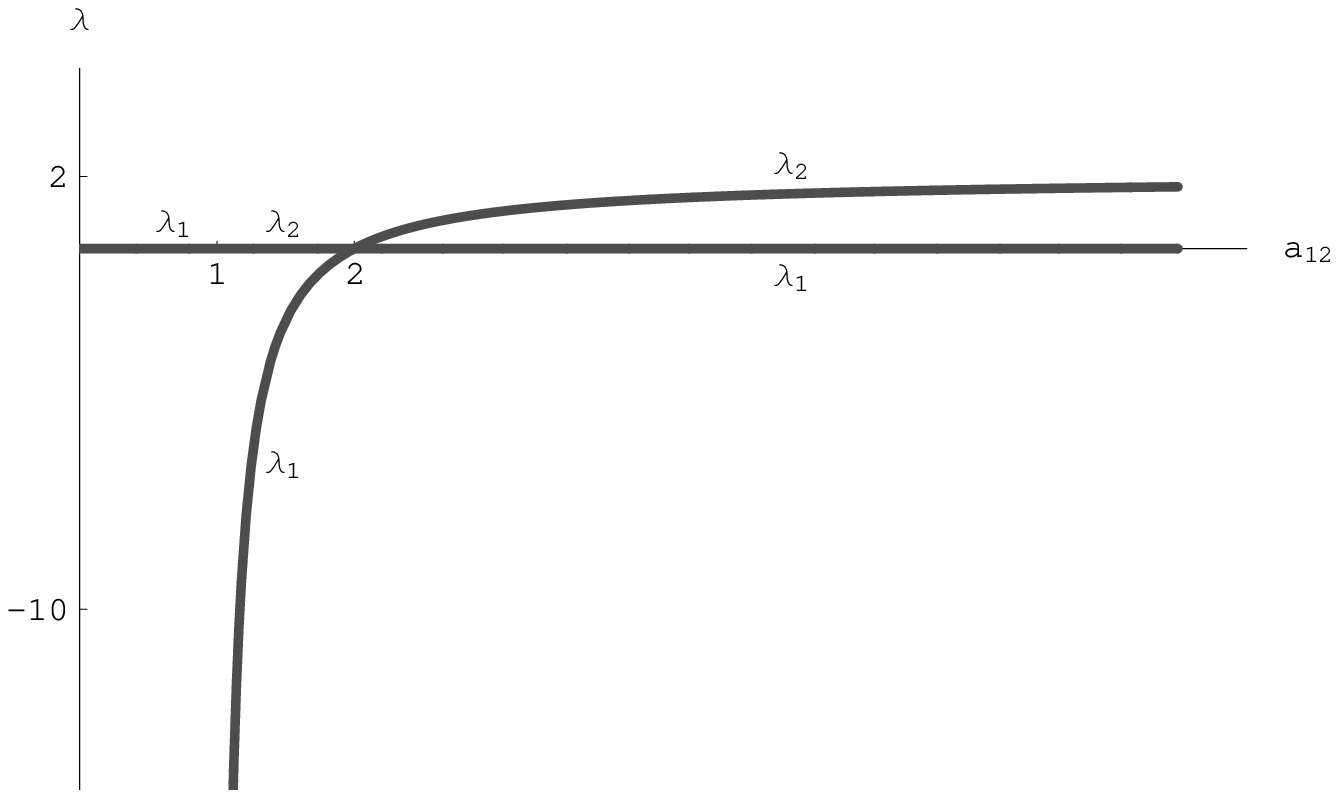}

{\small {\bf Figure 5.3.} One or two continuous eigenvalue branches, with horizontal parts}
\end{center}
\vspace{0.2cm}

Figure 5.3 also illustrates the monotonicity of continuous
eigenvalue branches in the $a_{12}$-direction in $\mathcal
O_{1,4}^{\mathbb C}$. Moreover, this example shows that the
monotonicity is not strict in general. \end{ex}

Finally, we give three examples to show different continuous eigenvalue branches on different subsets of $\Omega_{N}^{\mathbb R,+}$ of the discrete
self-adjoint SLP, separately.

\begin{ex}\label{e55} Let $s<0$. Consider the 1-parameter family of
self-adjoint discrete SLPs consisting of the discrete SLEs with
$$
f_0 =s, \ f_1 =1,\ f_2=1, \; q_1 =q_2 =0, \; w_1 =w_2 =1,\; N=2,
$$
and the BC
\begin{equation}\label{57}
\mathbf A =\left[\begin{array}{cccc} 1 & -1 & 0 & 1 \\ 0 & 1 & -1 & 0 \end{array}\right]
       \in \mathcal O_{1,3}^{\mathbb C}.
\end{equation}
Then, by Lemma \ref{l33}, direct calculations deduce that the
characteristic function is
$$
\Gamma(\lambda) =(-1+\lambda) (1+s\lambda)/s.
$$
Thus, the two continuous eigenvalue branches are
$$
\lambda_1(s) =\begin{cases} -1/s & \text{ if } s\le -1, \\
                     1 & \text{ if } -1 <s <0, \end{cases} \qquad
\lambda_2(s) =\begin{cases} 1 & \text{ if } s\le -1, \\
                     -1/s & \text{ if } -1 <s <0. \end{cases}
$$
See Figure 5.4. Therefore, in general, continuous
eigenvalue branches are not differentiable with respect to $f_n$,
and their monotonicity with respect to $f_n$ is not strict.

\begin{center}

\includegraphics[width=75mm]{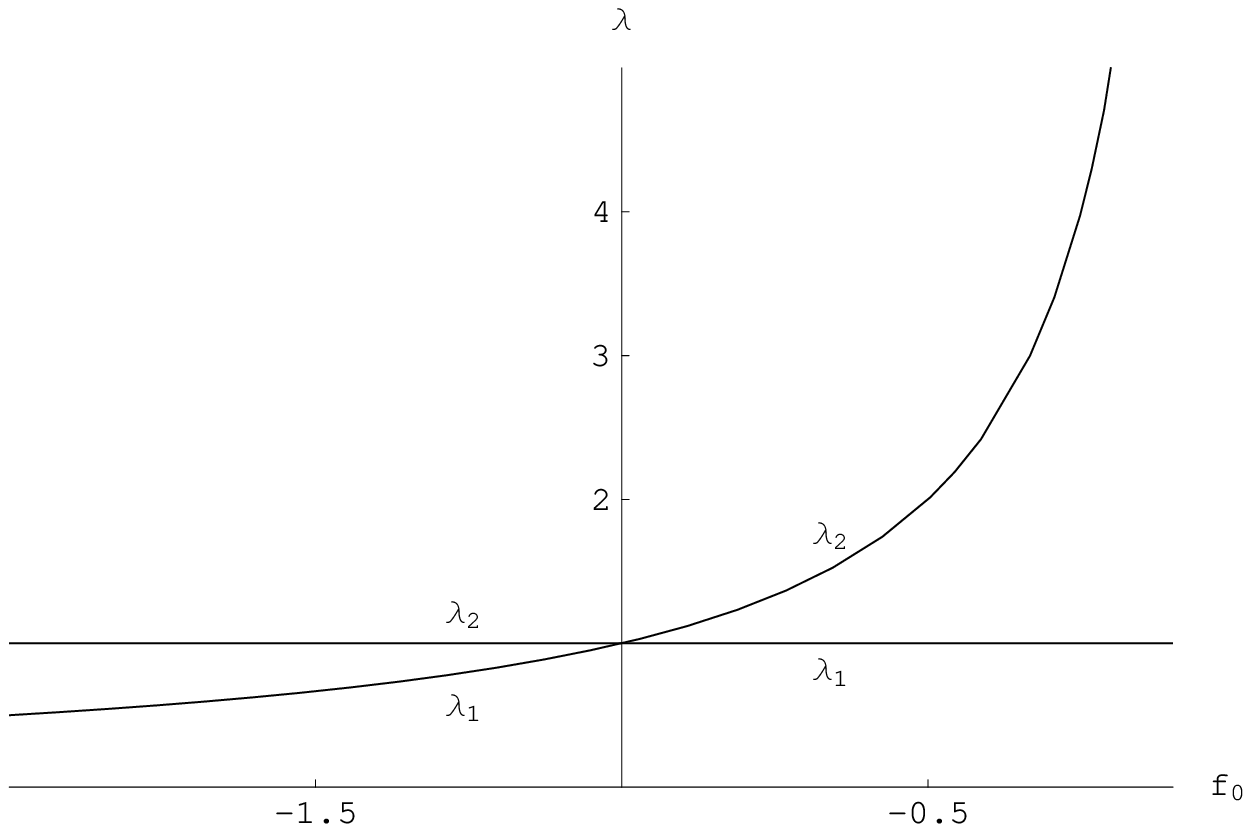}

{\small {\bf Figure 5.4.} Continuous eigenvalue branches are increasing in $f_0$-direction}
\end{center}

\end{ex}

\begin{ex}\label{e56}
 Let $s \in \mathbb R$. Take the 1-parameter family of
self-adjoint discrete SLPs consisting of the discrete SLEs with
$$
f_0 =-1, \ f_1 =1,\ f_2=1, \;\;\;\;q_1 =s, \ q_2 =0, \;\;\;\;w_1 =w_2 =1,\;\;\;\;N=2,
$$
and the BC given in \eqref{57}. Then, the characteristic function is
$$
\Gamma(\lambda) =(1-\lambda) (1+s-\lambda).
$$
Thus, the two continuous eigenvalue branches are
$$
\lambda_1(s) =\begin{cases} 1+s & \text{ if } s \le 0, \\
                     1 & \text{ if } s >0, \end{cases} \qquad
\lambda_2(s) =\begin{cases} 1 & \text{ if } s \le 0, \\
                     1+s & \text{ if } s >0. \end{cases}
$$
See Figure 5.5. Therefore,  in general, continuous
eigenvalue branches are not differentiable with respect to $q_n$,
and their monotonicity with respect to $q_n$ is not strict.

\begin{center}

\includegraphics[width=75mm]{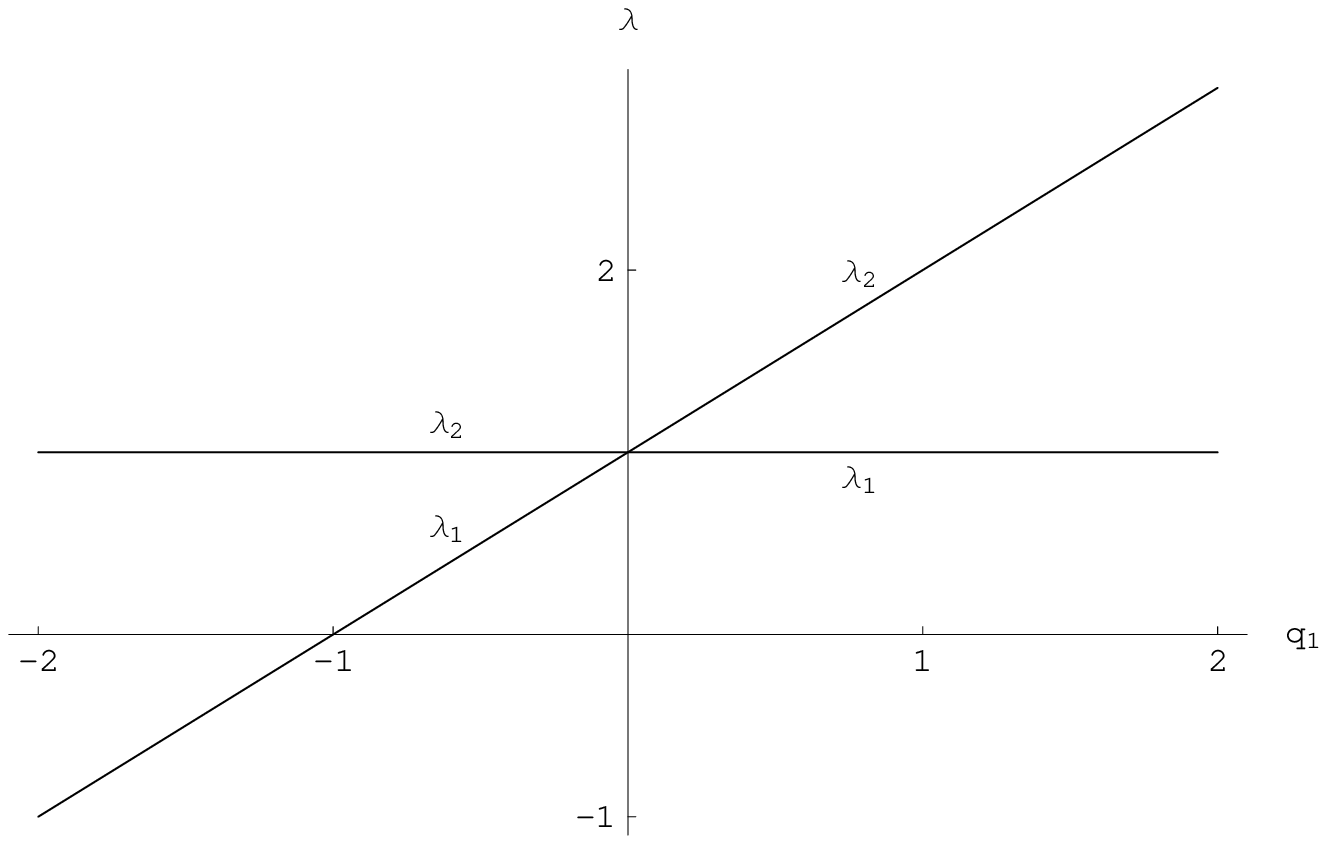}

{\small {\bf Figure 5.5.} Continuous eigenvalue branches are increasing in $q_1$-direction}
\end{center}

\end{ex}

\begin{ex}\label{e57} Let $s >0$. Consider the 1-parameter family of
self-adjoint discrete SLPs consisting of the discrete SLEs with\vspace{-0.01cm}
$$\vspace{-0.01cm}
f_0 =-1, \ f_1 =1,\ f_2=1, \;\;\;\; q_1 =q_2 =0,\;\;\;\; w_1 =s, \ w_2 =1,\;\;\;\; N=2,
$$
and the BC given in \eqref{57}. Then, the characteristic function is
$$
\Gamma(\lambda) =(1-\lambda) (1-s\lambda).
$$
Thus, the two continuous eigenvalue branches are
$$
\lambda_1(s) =\begin{cases} 1 & \text{ if } 0 <s \leq 1, \\
                     1/s & \text{ if } s >1, \end{cases} \qquad
\lambda_2(s) =\begin{cases} 1/s & \text{ if } 0 <s \leq 1, \\
                     1 & \text{ if } s >1. \end{cases}
$$
See Figure 5.6. Therefore, in general, continuous
eigenvalue branches are not differentiable with respect to $w_n$,
and the monotonicity of their positive parts with respect to $w_n$
is not strict.

\begin{center}
\includegraphics[width=75mm]{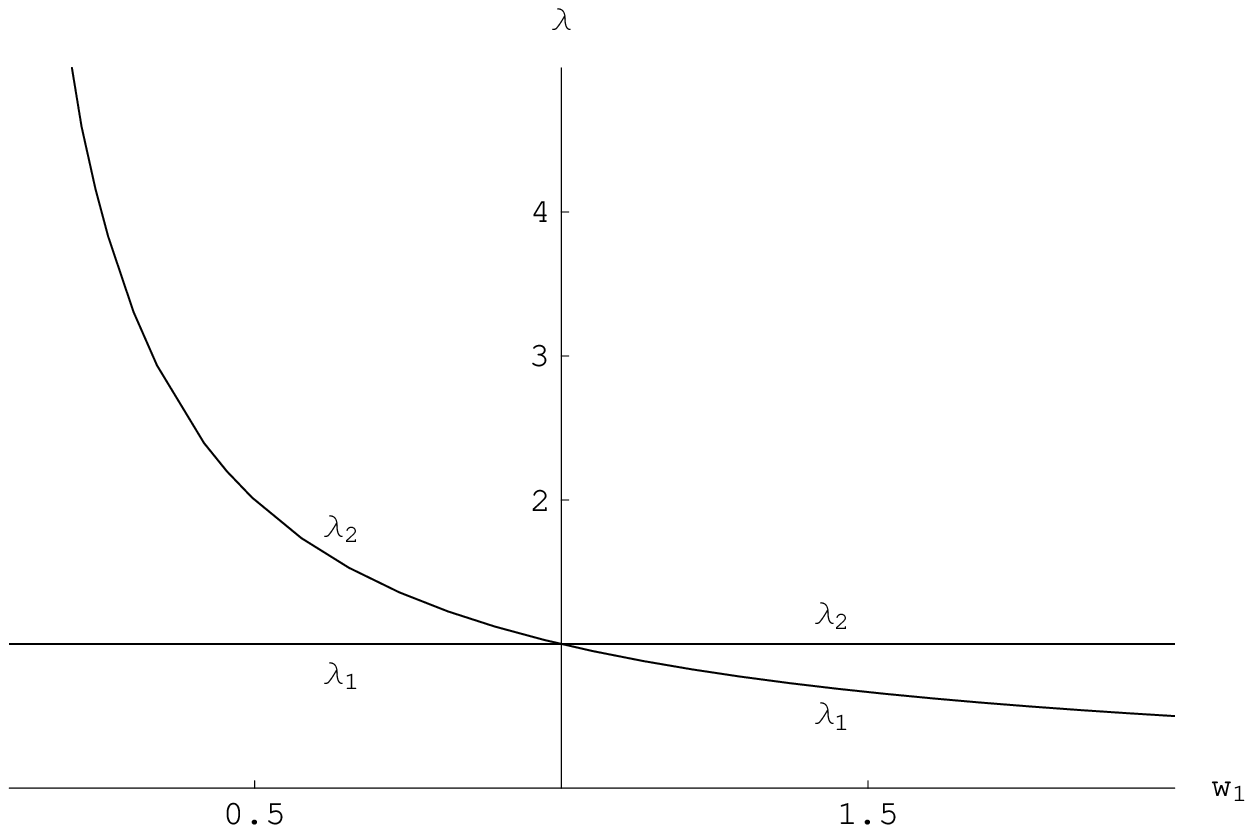}
\end{center}
{\small {\bf Figure 5.6.} Positive parts of continuous eigenvalue branches are decreasing in $w_1$-direction}

\end{ex}

\end{document}